\documentclass[a4paper,10pt]{article}

\usepackage{epsfig,amssymb, amsthm}
\usepackage{amsmath,amsfonts}
\usepackage[arrow, matrix, curve]{xy}

\begin{document}

\newcommand\R{\mathbb{R}}
\newcommand\bp{\mathbb{P}}
\newcommand\Z{\mathbb{Z}}
\newcommand\C{\mathbb{C}}
\newcommand\D{\Delta}

\newcommand\E{E}           
\newcommand\Ep{\overline{\mathcal{E}}} 
\newcommand\Eq{\hat{\mathcal{E}}}   
\newcommand\PP{\mathcal{P}}

\newcommand\Off{\mathop{\rm Off}}
\newcommand\Env{\mathop{\rm Env}}
\newcommand\Res{\mathop{\rm Res}}
\newcommand\sign{\mathop{\rm sign}}
\newcommand\dia{\lozenge}
\newcommand\pl{Pl$\ddot{\mbox{u}}$cker\ } 

\newcommand\Ec{\mathcal{E}}           
\newcommand\Gc{\mathcal{G}}
\newcommand\Gm{\Gamma}
\newcommand\Cc{\mathcal{C}}   
\newcommand\Sc{\mathcal{S}}   
\newcommand\Qc{\mathcal{Q}}   
\newcommand\Vc{\mathcal{V}}   
\newcommand\Hc{\mathcal{H}}
\newcommand\Mc{\mathcal{M}}

\newcommand\ee{{\mathcal{E}}}  
\newcommand\eep{\overline{e}}   

\newcommand\G{G}                
\newcommand\Gp{\overline{G}}    
\newcommand\Gq{\hat{G}}

\newcommand\F{F}
\newcommand\Hh{H}

\newcommand\V{\Vc(\Eq)}       

\newcommand\yy{y}                
\newcommand\yyp{\overline{y}}    
\newcommand\yyq{\hat{y}}         

\newcommand\Phip{{\Phi}}   

\theoremstyle{plain}
\newtheorem{thm}{Theorem}
\newtheorem{lem}[thm]{Lemma}
\newtheorem{cor}[thm]{Corollary}
\newtheorem{conjecture}[thm]{Conjecture}
\newtheorem{prop}[thm]{Proposition}

\theoremstyle{definition}
\newtheorem{defn}[thm]{Definition}
\newtheorem{rem}[thm]{Remark}
\newtheorem{assum}[thm]{Assumption}
\newtheorem{exmp}[thm]{Example}

\title{The implicit equation of a canal surface\footnote{This is the accepted authors' manuscript of the text. The article is to be published in the Journal of Symbolic Computation, see http://dx.doi.org/10.1016/j.jsc.2008.06.001 for more information.}}

\author{Marc Dohm, Severinas Zube}

\maketitle

\begin{abstract}
\noindent A canal surface is an envelope of a one parameter family of spheres.
In this paper we present an efficient algorithm for computing the
implicit equation of a canal surface generated by a rational family
of spheres. By using Laguerre and Lie geometries, we relate the
equation of the canal surface to the equation of a dual variety of
a certain curve in 5-dimensional projective space. We define the $\mu$-basis for arbitrary dimension and give a simple algorithm for
its computation.
This is then applied to the dual variety, which allows us to deduce
the implicit equations of the the dual variety, the canal surface and
any offset to the canal surface.\\

\end{abstract}

\section{Introduction}

In surface design, the user often needs to perform rounding or
filleting between two intersecting surfaces. Mathematically, the
surface used in making the rounding is defined as the envelope of
a family of spheres which are tangent to both surfaces. This
envelope of spheres centered at $c(t)\in\R^3$ with radius $r(t)$, where
$c(t)$ and $r(t)$ are rational functions, is called a canal surface with spine curve
$\Ec=\{(c(t),r(t))\in\R^4 | t \in \R \}$. If the radius $r(t)$ is constant the
surface is called a pipe surface. Moreover, if additionally we
reduce the dimension (take $c(t)$ in a plane and consider circles
instead of spheres) we obtain the offset to the curve. Canal
surfaces are very popular in Geometric Modelling, as they can be
used as a blending surface between two surfaces. For example, any
two circular cones with a common inscribed sphere can be blended
by a part of a Dupin cyclide bounded by two circles as it was
shown by \cite{Pratt90,Pratt95}
(see Figure \ref{fig:cyclide}).
Cyclides are envelopes of special quadratic families of spheres.
For other examples of blending with canal surfaces we refer to
\cite{margarita}.

\begin{figure}
\begin{center}
\epsfxsize=3.5cm  \epsfbox{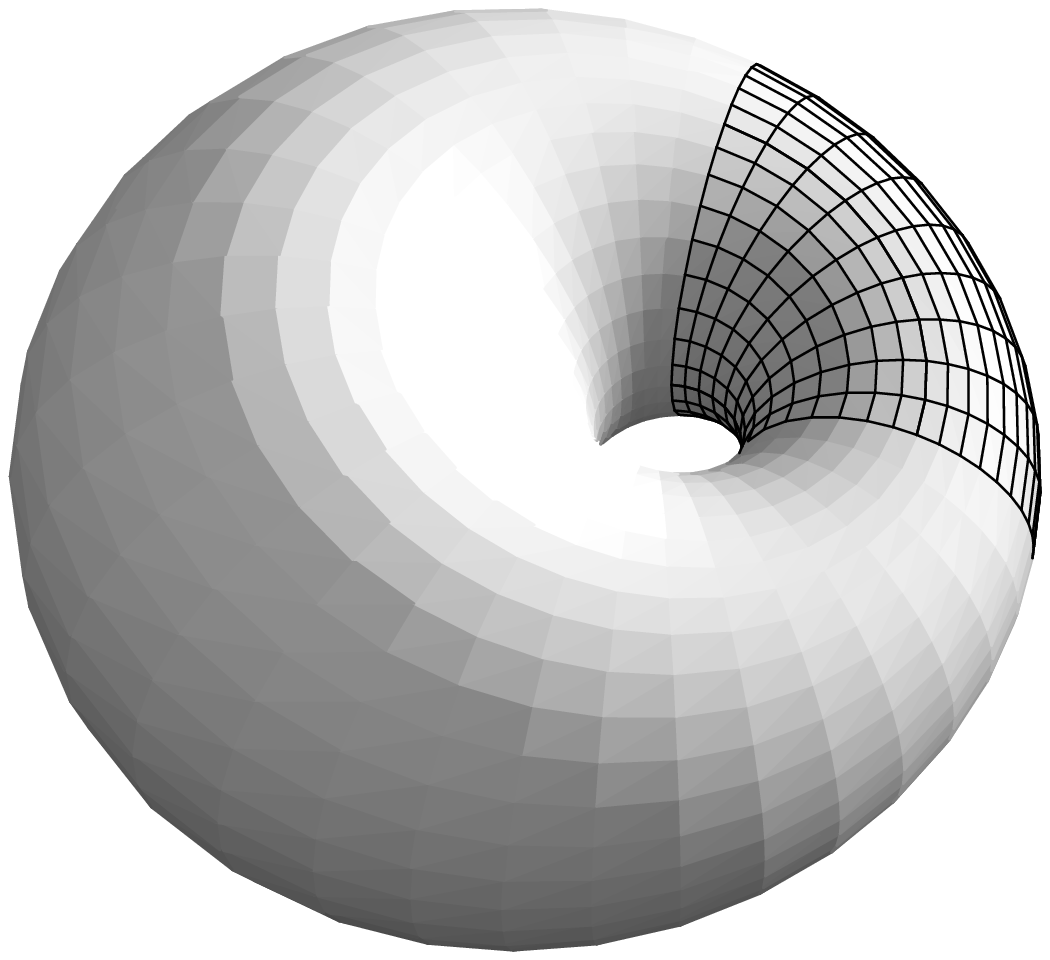} \hspace{1cm}
\epsfxsize=3cm \epsfbox{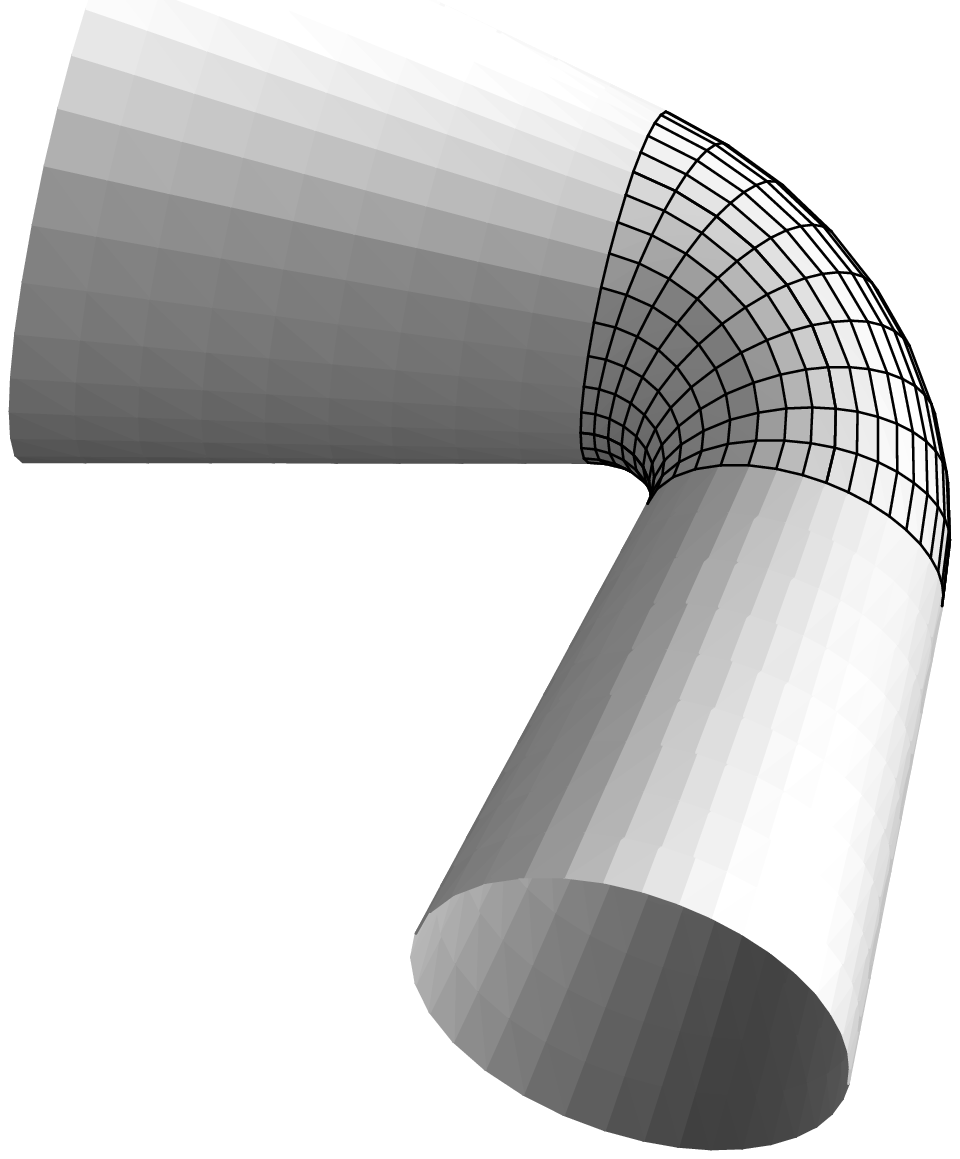}
\caption{\label{fig:cyclide} A Dupin cyclide used for blending
circular cones.}
\end{center}
\end{figure}

Here we study the implicit equation of a  canal surface $\Cc$ and
its implicit  degree. The implicit equation of a canal surface can
be obtained after elimination of the family variable $t$ from the
system of two equations $g_1(\yy,t)=g_2(\yy,t)=0$ (here $g_1,g_2$
are quadratic in the variables
 $\yy=(y_1,y_2,y_3,y_4)$), i.e. by
taking the resultant with respect to $t$. However, this resultant
can have extraneous factors. In the paper we explain how these
factors appear and how we can eliminate them. By using Lie and
Laguerre geometry, we see that the above system of equations  is
related to a system $h_1(\yyq,t)=h_2(\yyq,t)=Q(\yyq)=0$, where
$h_1,h_2$ are linear in the variables
$\yyq=(u,y_0,y_1,y_2,y_3,y_4)$ and $Q(\yyq)$ is the Lie quadric
(for the exact definition see formula (\ref{liequadric})). It
turns out that the variety defined by the system of equations
$h_1(\yyq,t)=h_2(\yyq,t)=0$ is a dual variety to the curve $\Eq
\in \bp^5$, where $\Eq$ is a curve on the Lie quadric determined
by the spine curve $\Ec$ (for the explicit definition see formula
(\ref{bar_e})). For the dual variety $\V$ we define the
$\mu$-basis, which consists of two polynomials
$p_1(\yyq,t),p_2(\yyq,t)$ which are linear in $\yyq$, and of
degree $d_1,d_2$ in $t$ are such that $d_1+d_2$ is minimal. It
turns out that the  resultant of $p_1$ and $p_2$ with respect to
$t$ gives the implicit equation of the variety $\V$. There is a
simple substitution formula (see the algorithm  at the end of
section \ref{dualvariety}) to compute the implicit equation of the
canal surface from the implicit equation of the variety $\V$.

Partial  solutions to the problem of finding the implicit equation (and degree) for canal
surfaces have been given in other papers. For instance, the degree of offsets to curves is
studied in \cite{ss}. In \cite{xu},
there is a degree  formula for the implicit equation of a
polynomial canal surface. Quadratic canal surfaces (parametric and
implicit representation) have been studied in \cite{KZ}.

We close the introduction by noting that the implicit degree of a
canal surface is  important for the parametric degree. Our
observation is that if the canal surface has the minimal
parametrization of bi-degree $(2,d)$ then its implicit degree is
close to $2d$. On the minimal    bi-degree $(2,d)$
parametrizations of the canal surface we refer to \cite{kra}.

The paper is organized as follows.  In the next
section we develop some algebraic formalism about modules with two
quasi-generators.  We define the $\mu$-basis for these modules and
present an algorithm for its computation. In the following section,
we recall some needed facts about Lie and Laguerre sphere geometry. Then using Lie and
Laguerre geometry we describe the canal surface  explicitly. Also,
we introduce the $\Gm$-hypersurface which contains all $d$-offsets
to the canal surface.  Using the $\mu$-basis algorithm  we compute
the implicit equations of the dual variety $\V$, the
$\Gm$-hypersurface  and the canal surface $\Cc$. Next we apply the
results of the previous section to the dual variety $\V$ of the
curve  and explain how to compute the implicit degree
of the  $\Gm$-hypersurface (without computation of the implicit
equation). Finally, we give some computational examples.

\section{Modules with two quasi-generators and the $\mu$-basis.} \label{quasigens}
\vskip0.1cm

Let $\R[t]$ be polynomial ring over the field of real numbers, and
denote $\R[t]^d$ the $R$-module of $d$-dimensional row vectors with
entries in $\R[t]$. Let $\R(t)$ be the field of rational functions
in $t$. For a pair of vectors $A=(A_1,A_2,\ldots,A_d)$,$B=(B_1,B_2,...,B_d) \in\R[t]^d$ the set
\begin{equation}\label{module}
M=\langle A,B \rangle=\{aA+bB\in\R[t]^d\ |\
a,b\in\R(t),A,B\in\R[t]^d\}\subset\R[t]^d
\end{equation} is the $\R[t]$-module with two polynomial \textit{quasi-generators} $A,B$. Here, we
assume that $A,B$ are $\R[t]$-linearly independent, i.e. $aA+bB=0$
with $a,b\in\R[t]$ if and only if $a=b=0$.

\noindent \textbf{Remark:} Note that the vectors $A,B$ may not be
generators of the module $M$ over $\R[t]$ because $a$ and $b$ in the
definition (\ref{module})  are from the field $\R(t)$
of rational functions. For example, if $A=p
D$ with $p\in\R[t],D\in\R[t]^d$ and $\deg p > 0$ then $A,B$ are not
generators of the module $M$.\\

\noindent For  $A=(A_1,A_2,\ldots,A_d)$,$B=(B_1,B_2,...,B_d) \in\R[t]^d$
we define the \pl coordinate vector $A \wedge B$ as follows:
\begin{equation}
    A\wedge B=([1,2],[1,3],...,[d-1,d])\in\R[t]^{d(d-1)/2},\ \mbox{where}\
    [i,j]=A_iB_j-A_jB_i.  \nonumber
\end{equation}
In other words, $A \wedge B$ is the vector of $2$-minors of the matrix
$$W_{A,B}=\begin{pmatrix} A_1 & A_2 &  \cdots & A_d \cr B_1
& B_2 &  \cdots & B_d \end{pmatrix}$$ and we denote by
$\deg(A\wedge B)=\max_{i,j}\{\deg(A_iB_j-A_jB_i)\}$ the degree of
the \pl coordinate vector, i.e. the maximal degree of a 2-minor of
$W_{A,B}$.

Let a polynomial vector $A\in\R[t]^d$ be presented as
\begin{equation}
    A=\sum_{i=0}^n \alpha_it^i,\ \alpha_i\in\R^d,\ i=0,...,n;\ \alpha_n\not=0. \nonumber
\end{equation}
We denote the leading vector $\alpha_n$ by $LV(A)$ and the degree of
$A$ by $\deg A=n$.

  Note that if $LV({A})$ and $LV({B})$ are linearly independent over $\R$ then $\deg A\wedge B
=\deg A+\deg B$
 and $LV(A\wedge B)=LV(A)\wedge LV(B)$. We define
 $$\deg M=\min\{\deg
(\tilde A\wedge \tilde B)\ |\ \mbox{ }\  \tilde A,\tilde B
\in\R[t]^d \ \ \mbox{such that}\  \langle \tilde A,\tilde B \rangle=M
\}$$
 to be the degree of the module $M$ with two
 quasi-generators.\medskip

\begin{defn}
 Two quasi-generators $\tilde A,\tilde B$ of the module $M=\langle
A,B\rangle$ are called a $\mu$-basis of the module $M$ if $\deg
M=\deg\tilde A +\deg \tilde B$.
\end{defn}

 As we always have the inequality $\deg
( A\wedge  B) \leq \deg A +\deg B$, this means in particular that
the sum $\deg\tilde A +\deg \tilde B$ is minimal. A $\mu$-basis always
exists, as we shall see at the end of the section. Let us explain
the geometric motivation behind this definition.

\begin{rem} By abuse of notation, we will continue to denote parameters $t$, however in
the geometric definitions that follow, they should be understood
as parameters $(t:s) \in \bp^1$ and polynomials in $\R[t]$ should
be thought of as homogenized with respect to a new variable~$s$.
\end{rem}

We define the following subspace of $\R^d$ for the module $M=\langle A,B\rangle$.
     $$L(M,t_0) =\{x\in\R^d\ |\ C(t_0)\cdot x=0 \ \mbox{for all}\ C\in M\}$$
where $C(t)=(C_1(t),C_2(t),\ldots,C_d(t)) \in \R[t]^d$, $x=(x_1,x_2,\ldots,x_d)^T$ and $C(t) \cdot x=x_1C_1(t)+x_2C_2(t)+\ldots+x_dC_d(t)$.  We have the inequality $\dim L(M,t_0)\geq d-2$, because the module $M$ has
only two quasi-generators. In fact, we have 
$\dim L(M,t_0)= d-2$ for all $t_0$, as we will see in Proposition \ref{mu-basis-prop}.2.
Whenever two vectors $A(t_0)$ and
$B(t_0)$ are linearly independent in $\R^d$ then
$L(M,t_0)$ is the intersection of two hyperspaces $\{x\in\R^d |\
A(t_0)\cdot x =0\}$ and $\{x\in\R^d\ |\  B(t_0)\cdot x =0\}$.

Using those subspaces, we can associate a hypersurface $\Sc_{M}$ in the real projective space
$\bp^{d-1}=\bp(\R^d)$ with the module $M$

\begin{equation}\label{druled}
    \Sc_{M}:=\bigcup_t \bp(L(M,t))\subset\bp^{d-1}.
\end{equation}

Note that this definition and the definition of $L(M,t_0)$ depend only on the module $M$ and
not on the choice of quasi-generators. It is useful to compare the hypersurface $\Sc_M$ with the
hypersurface $\Sc_{A,B}$ defined as
\begin{equation}\label{druledAB}
    \Sc_{A,B}:=\bigcup_t (\{A(t)\cdot x \} \cap \{ B(t)\cdot x \})\subset\bp^{d-1}
\end{equation}
where $ A, B$ are quasi-generators of $M$. By definition, this is the variety
defined by $\Res{}_t(A(t)\cdot x,B(t)\cdot x)$ and it is clear that
$\Sc_{M}\subset \Sc_{A,B}$. If the vectors $A(t_0),B(t_0)$
are linearly dependent, then $(\{ A(t_0)\cdot x \} \cap \{B(t_0)\cdot x \})\subset\R^d$ is a subspace of codimension one.  Note that
in this case the implicit equation $\Res{}_t(A(t)\cdot x,B(t)\cdot x)$ contains the factor $A(t_0) \cdot x$.
As a matter of fact, this happens if and only if $W_{A,B}(t_0)$ has rank one, which is equivalent to
saying that $t_0$ is a zero of the ideal generated by the \pl coordinates.

In fact, we will see in Proposition \ref{s_m} that this phenomenon does not occur
for $\mu$-bases, i.e. if $\tilde{A},\tilde{B}$ is a $\mu$-basis of the module $M$ then
$\Sc_{M}=\Sc_{\tilde{A},\tilde{B}}$ and there are no extraneous factors as before.

\begin{rem}
We should explain why we use the term $\mu$-basis. The above
definition is a generalization of the usual definition for  the
$\mu$-basis of a rational ruled surface (as in \cite{csc},
\cite{czs} or \cite{dohm}). They coincide in the special case
$d=4$. $M$ is the analogue of the syzygy module (i.e. the module
of moving planes following the parametrization of the ruled
surface) and the subspaces  $L(M,t)$, which in this case are
2-dimensional and hence define projective lines, are exactly the
family of lines which constitute the ruled surface. Similarly, the
case $d=3$ corresponds at the theory of $\mu$-bases for rational
curves and our definition is equivalent to the usual definition as
in \cite[Theorem 3, Condition 3]{cw}.

However, the approach used here is actually inverse to the approach in the cited papers. In the latter
the ruled surface is defined by a parametrization and then the module of moving planes is studied, whereas
here we fix a module that ``looks like'' such a moving plane module and then study the (generalized) ruled surface
that corresponds to it. Note that by  definition of the subspaces $L(M,t)$ any element $C$ of $M$ can be considered a moving plane following
$\Sc_M$, in the sense that for all $x \in \Sc_M$ there is a parameter $t$ such that $ C(t)\cdot x=0$.

Note that $A \wedge B$ defines the so-called Pl\"ucker curve $\mathcal P$ in $\bp^{d(d-1)/2 -1}$ by
$$\begin{array}{rcc}
  \varphi_\mathcal{P}:   \bp^1  & \dashrightarrow& \bp^{d(d-1)/2 -1}
      \nonumber \\
     t &\mapsto& ([1,2]:[1,3]:...:[d-1,d])
\end{array}$$
where $[i,j]=A_iB_j-A_jB_i$. We will denote $k=\deg
\varphi_\mathcal{P}$ the degree of the parametrization, which is
the cardinality of the fiber of a generic point in the image of
$\varphi_\mathcal{P}$. Note that $\varphi_\mathcal{P}$ and $k$ are
the same for any choice of quasi-generators of $M$.
\end{rem}

\begin{prop}\label{degreeofSM}
For any pair of quasi-generators $A,B$ of $M$ we have the degree formula
$$k \cdot \deg \Sc_M  
=\deg (A\wedge B)-\deg q_{A,B},$$ where $q_{A,B}=\gcd(A\wedge B)$
and $k=\deg \varphi_\mathcal{P}$. Moreover, we have $\deg
\mathcal{P}=\deg \Sc_M$.
\end{prop}
\begin{proof}
The proposition and the proof are similar to  Lemma 1 in
\cite{czs} and to Theorem 5.3 in \cite{PPR98}.

 The implicit degree of the  hypersurface $ \Sc_{A,B}$ is the
 number of intersections between a generic line and the  hypersurface.
 The generic line $L(s)$  is defined by two points in the space
  $L(s)=H_0+ s H_1$, where
 $H_i=(h_{i1},h_{i2},...,h_{id}),i=0,1$.
 The line $L(s)$ intersects the hyperplane $\{A(t)\cdot x \}$ if and only if $H_0\cdot A(t)+sH_1\cdot
 A(t)=0$. Since the line $L(s)$ should intersect the hyperplane $\{ B(t)\cdot x \}$ too, we see
 that the  implicit degree is the number of intersections of two curves in the $(t,s)$ plane:
    $$H_0\cdot A(t)+sH_1\cdot A(t)=0,$$ 
    $$H_0\cdot B(t)+sH_1\cdot B(t)=0.$$ 
Eliminating $s$ from the above equation we have
\begin{eqnarray}\label{ww}
\left|\begin{array}{cc}
    H_0\cdot A(t) & H_1\cdot A(t)\\
    H_0\cdot B(t)& H_1\cdot B(t)
    \end{array}\right|=(H_0\wedge H_1)\cdot (A(t)\wedge
    B(t))=0,
\end{eqnarray}
where  $C\cdot D$ means a standard scalar product of two vectors
$C,D\in\R^{d(d-1)/2}$. The number of solutions of \eqref{ww} is
the number of intersection points of the Pl\"ucker curve with a
generic hyperplane in $\bp^{d(d-1)/2 -1}$, so $\deg \mathcal{P}=\deg \Sc_M$.

 Now it is known (see for example \cite[Theorem 1]{dohm}) that $$k \cdot \deg \mathcal{P}  =\deg (A\wedge B)-\deg q_{A,B}$$ and
the proposition follows.\end{proof}

We have yet to show the existence of the $\mu$-basis. To this end, we propose
 an algorithm for its computation,
the basic idea of which is to reduce $q_{A,B}=\gcd(A\wedge B)$ to a
constant using the so-called Smith form of the $2 \times d$ matrix
$$W_{A,B}=\begin{pmatrix} A_1 & A_2 &  \cdots & A_d \cr B_1
& B_2 &  \cdots & B_d \end{pmatrix}$$ and then render the leading
vectors linearly independent by a simple degree reduction. The
Smith form is a decomposition $W_{A,B}=U \cdot S \cdot V$, with
unimodular $U \in \mathbb{R}[t]^{2 \times 2}$, $V \in
\mathbb{R}[t]^{d \times d}$, and
$$S= \begin{pmatrix} 1 & 0 & 0 & \cdots & 0 \cr 0
& q_{A,B} & 0 & \cdots & 0 \end{pmatrix} \in \mathbb{R}[t]^{2 \times d}$$
It always exists and can be computed efficiently by standard computer algebra systems.\\

\textbf{Algorithm ($\mu$-basis)}
\begin{enumerate}
\item INPUT: Quasi-generators $A=(A_1,A_2,\ldots,A_d)$,$B=(B_1,B_2,...,B_d) \in\R[t]^d$ of the
module $M$

 \item Set $$W_{A,B}=\begin{pmatrix} A_1 & A_2 &  \cdots & A_d \cr B_1
& B_2 &  \cdots & B_d \end{pmatrix}.$$

 \item Compute a Smith form $$W_{A,B}=U \cdot \begin{pmatrix} 1 & 0 & 0 & \cdots & 0 \cr 0
& q_{A,B} & 0 & \cdots & 0 \end{pmatrix} \cdot V$$ with unimodular $U \in
\mathbb{R}[t]^{2 \times 2}$,$V \in \mathbb{R}[t]^{d \times d}$.

\item Set $W'$ to be the $2 \times d$-submatrix consisting of the
first two rows of $V$. \item If the vector of leading terms (with
respect to the variable $t$) of the first row is $h$ times the one
of the second row, $h \in \R[t]$, set $W':= \begin{pmatrix} 1 & -h \cr 0 & 1 \end{pmatrix} \cdot
W'$.

\item If the vector of leading terms (with respect to the variable
$t$) of the second row is $h$ times the one of the first row, $h
\in \R[t]$, set $W':= \begin{pmatrix} 1 & 0 \cr -h & 1 \end{pmatrix} \cdot W'$.

\item If the preceding two steps changed $W'$ go back to Step 5.

\item Set $\tilde A,\tilde B$ to be the rows of $W'$.

\item OUTPUT: A $\mu$-basis $\tilde A,\tilde B$ of the module $M$
\end{enumerate}
\textbf{ }\\
As we shall see in Section \ref{dualvariety}, the case we are
interested in is the case $d=6$, so we are dealing with very small
matrices and the computations are extremely fast. Note that we
actually only need the first two rows of $V$, so  we could optimize
the algorithm by modifying the Smith form algorithm used as not to
compute the unnecessary entries of the matrices $U$ and $V$.
Generally, the number of elementary matrix operations in Step 5
and 6 is very low. In the worst case, it is bounded by the maximal
degree of the entries of the matrix $W'$ in Step 4 of the
algorithm, since each step reduces the maximal degree in one of
the rows of $W'$.

Next, we will show that the output of the above algorithm is a $\mu$-basis and that the resultant of a $\mu$-basis $\tilde A,\tilde B$ of the module $M=\langle A,B\rangle$
is an implicit equation of $\Sc_M$. In Section \ref{dualvariety}, we will use these results for
a special choice of $A$ and $B$ to compute the implicit equation of a canal surface.

\begin{lem} \label{mubadeg}
The output of the above algorithm is a $\mu$-basis and we have $k \cdot \deg \Sc_M=\deg M$, where
$k=\deg \varphi_\mathcal{P}$.
\end{lem}

\begin{proof}
Let $\tilde A(t),\tilde B(t)$ be the output of the above algorithm.
By construction it is clear that $\tilde A(t),\tilde B(t)$ are
quasi-generators of $M$ and that $\tilde q_{A,B} = \gcd(\tilde A\wedge \tilde
B)=1$. Furthermore, we have $\deg(\tilde A \wedge \tilde B) =\deg(\tilde
A)+\deg(\tilde{B})$, because the vectors of leading terms of $\tilde
A(t)$ and $\tilde B(t)$ are linearly independent. So by
Proposition~\ref{degreeofSM} we deduce
\begin{eqnarray}
k \cdot \deg \Sc_M &=  &  \deg(\tilde A \wedge \tilde B) - \deg(\tilde q) \nonumber \\
&=& \deg(\tilde A \wedge \tilde B) \nonumber \\
& = & \deg(\tilde A)+\deg(\tilde B) \nonumber
\end{eqnarray}
Moreover, by definition we have $\deg M \leq \deg(\tilde A \wedge
\tilde B)$ and if $A,B$ are quasi-generators such that $\deg (A
\wedge B)$ is minimal, the degree formula gives $\deg(\tilde A
\wedge \tilde B) = \deg( A \wedge  B) - \deg( q) \leq \deg M$,
which shows that $k \cdot \deg \Sc_M=\deg M$, and as a consequence
that $\tilde A(t),\tilde B(t)$ is indeed a $\mu$-basis.
\end{proof}

\begin{prop}\label{s_m}
Let $\tilde A(t),\tilde B(t)$  be a  $\mu$-basis of $M$ and $x=(x_1,x_2,\ldots,x_d)^T$ variables. Then
$$\Res{}_t(\tilde{A}(t)\cdot x,\tilde{B}(t)\cdot x)=F_{\Sc_M}^k$$ where $F_{\Sc_M}$ is the implicit
equation of the hypersurface $\Sc_M$.
\end{prop}

\begin{proof}
First, we will show in the same way as in \cite[Theorem 9]{dohm}
that $\Res{}_t(\tilde{A}(t)\cdot x,\tilde{B}(t)\cdot x)$ is geometrically
irreducible, i.e. the power of an irreducible polynomial.
As we shall see in Proposition \ref{mu-basis-prop}, the intersection of the
hyperplanes $\{\tilde{A}(t)\cdot x \}$ and $\{\tilde{B}(t)\cdot x \}$ is of
codimension 2 for any parameter $t \in \bp^1$. So the incidence variety
$$\mathcal W = \{ (t,x) \in \bp^1 \times \bp^{d-1} \vert \tilde{A}(t)\cdot x= \tilde{B}(t)\cdot x =0 \} \subset \bp^1 \times \bp^{d-1}$$
is a vector bundle over $\bp^1$ and hence irreducible. So the projection on $\bp^{d-1}$ is irreducible as well and
its equation, which is by definition the hypersurface defined by $\Res{}_t(\tilde{A}(t)\cdot x,\tilde{B}(t)\cdot x)$, is a
power of an irreducible polynomial.

As we have remarked earlier, the resultant of two quasi-generators
is always a multiple of the implicit equation of $\Sc_M$, so $\Res{}_t(\tilde{A}(t)\cdot x,\tilde{B}(t)\cdot x)$
is a power of $F_{\Sc_M}$.

 But using the degree property above we see  $$ \deg(\Res{}_t(\tilde{A}(t)\cdot x,\tilde{B}(t)\cdot x)) =\deg(\tilde A)+\deg(\tilde B)=k \cdot \deg \Sc_M$$ which implies that $\Res{}_t(\tilde A(t)\cdot x,\tilde
B(t)\cdot x)$ equals $F_{\Sc_M}^k$.
\end{proof}

\begin{rem}
It is known that the Pl\"ucker curve $\mathcal P$ can be properly reparametrized,
i.e. there exists a rational function $h$ of
degree $k$ such that $A \wedge B = C \circ h$, where $C$ is a proper parametrization of $\mathcal P$.
It is tempting to use
this proper reparametrization in order to represent the implicit
equation $F_{\Sc_M}$ of $\Sc_M$ directly as a resultant as in the proof of
\cite[Theorem 3]{dohm}. However,
$h$ does not necessarily factorize $A$ and $B$, i.e. it is not sure
that there exist $A'$ and $B'$ with $A=A'\circ h$ and $B=B'\circ h$, which would be needed to do this.
\end{rem}

In the following we present some properties of  $\mu$-bases. Note
that the properties in Propositions~\ref{mu-basis-prop},~\ref{mu-basis-prop-2}
are  similar  to \cite{cw} Theorems 1,3. However, we give different proofs by deducing them from the
degree formula and Lemma~\ref{mubadeg}.

\begin{prop}\label{mu-basis-prop}
Let $M=\langle A,B\rangle$ and let $\tilde{A},\tilde{B}$ be a $\mu$-basis
of the module $M$. Then the following properties hold: \\
1. The vectors $LV(\tilde{A}),LV(\tilde{B})$ are linearly independent.\\
2. $\tilde{A}(t_0),\tilde{B}(t_0)$ are linearly independent over
$\mathbb{C}$ for any parameter value $t_0 \in \mathbb{C}$.
\end{prop}
\begin{proof}
1. If $LV(\tilde{A}),LV(\tilde{B})$ were linearly dependent, this would imply that
$k \cdot \deg \Sc_M=\deg(\tilde A \wedge \tilde B) - \deg(q_{\tilde A,\tilde B}) < \deg(\tilde A)+\deg( \tilde B) = \deg M$
which is a contradiction to Lemma~\ref{mubadeg}. \\
2. Suppose that $\tilde{A}(t_0),\tilde{B}(t_0)$ are linearly dependent
for some $t_0 \in \mathbb{C}$. This is equivalent to saying that
the matrix $W_{\tilde A,\tilde B}$ is not
of full rank, which means that all 2-minors vanish. So
$t_0$ is a root of $q_{\tilde A,\tilde B}$ and as above
we deduce $k \cdot \deg \Sc_M=\deg(\tilde A \wedge \tilde B) -
\deg(q_{\tilde A,\tilde B}) < \deg(\tilde A)+\deg( \tilde B) = \deg M$ which is again a
contradiction to Lemma~\ref{mubadeg}. \end{proof}

\begin{prop}\label{mu-basis-prop-2}
Let $M=\langle\tilde  A,\tilde B\rangle$ and  assume that
$\tilde{A},\tilde{B}$ satisfy conditions 1,2 from
Proposition~\ref{mu-basis-prop}. Then any element $D\in M$ has the
following expression: $D=h_1\tilde A+h_2\tilde B$ for some
$h_1,h_2\in\R[t]$, i.e. $\tilde{A},\tilde{B}$ are generators of the
module $M$ over the polynomial ring $\R[t]$. Moreover, the pair
$\tilde{A},\tilde{B}$ is a $\mu$-basis of the module $M$.
\end{prop}
\begin{proof}
Let $D \in M$, it can be expressed as
$$ D= \frac{a}{b} \tilde A + \frac{c}{d} \tilde B $$ with $a,b,c,d \in
\R[t]$ and co-prime numerators and denominators in the rational
functions $\frac{a}{b}$ and $\frac{c}{d}$. Furthermore, we may
assume that $\gcd(a,c)=1$, because if $\frac{D}{\gcd(a,c)}$ is a
linear combination of $\tilde{A},\tilde{B}$, then so is $D$.
Multiplying both sides of the above equation with $bd$ we obtain
$bd D= ad \tilde A + bc \tilde B$ or equivalently $b(d D-c \tilde B)= ad
\tilde A$ and since $b$ divides neither
 $a$ nor $\tilde A$ (if it divided $\tilde A$, for any root $t_0$ of $b$ and
any constant $\alpha$ we would deduce the relation $0=\alpha \tilde
A(t_0) + 0 \cdot \tilde B(t_0)$, which contradicts property 2 in
Proposition \ref{mu-basis-prop}), one concludes that $b$ divides
$d$ and by a symmetric argument that $d$ divides $b$, so we may
assume $b=d$. So we have $$ b D = a \tilde A + c \tilde B $$ and
plugging a root $t_0$ of $b$ into the equation, we would obtain a
non-trivial linear relation between $\tilde A$ and $\tilde B$,  again a contradiction to
Proposition \ref{mu-basis-prop}. This implies that $b$ and $d$ are
constant, which shows that any $D \in M$ can be expressed as
linear combination of $\tilde A$ and $\tilde B$ over $\R[t]$. In other
words: $\tilde A$ and $\tilde B$ are not only quasi-generators of $M$,
but actually generators in the usual sense, i.e. over $\R[t]$.

 Suppose that $\deg\tilde{A}\leq \deg\tilde{B}$ and let $M=\langle P_1,P_2 \rangle$.
Then we proved that $P_i=h_{i1} \tilde{A}+h_{i2}\tilde{B},i=1,2$ for
some polynomials $h_{ij}\in\R[t]$. Since $LV(\tilde{A}),LV(\tilde{B})$
are linearly independent  $LV(h_{i1}\tilde{A})$ and
$LV(h_{i2}\tilde{B}),i=1,2$ do not cancel each other. Therefore,
$\deg P_1\geq\deg\tilde B$ (if $h_{12}\not=0$) or $\deg
P_2\geq\deg\tilde B$ (if $h_{22}\not=0$). Also $\deg P_1\geq\deg\tilde
A$ and $\deg P_2\geq\deg\tilde A$. So, we see that $\deg P_1+\deg
P_2\geq\deg\tilde A+\deg\tilde B$, i.e. a pair $\tilde{A},\tilde{B}$ is a
$\mu$-basis of the module $M$.
\end{proof}

\section{Elements of Lie and Laguerre  Sphere Geometry}

Here we shortly recall the elements of Lie and  Laguerre Sphere
Geometry (cf. \cite{cecil,PP98,KM00}). We start from the
construction of Lie's geometry of oriented spheres and planes in
$\R^3$. Let $\textbf{p}\in \R^3$, $r\in\R$. The oriented sphere  $S_{\textbf{p},r}$ in $\R^3$ is the set
\begin{equation} 
    S_{\textbf{p},r}= \{\textbf{v}\in \R^3 | (\textbf{v}-\textbf{p})\cdot(\textbf{v}-\textbf{p})=r^2 \}, \nonumber
\end{equation}
where by $\textbf{v}\cdot \textbf{w}$ we denote the standard positive definite
scalar product in $\R^3$. The orientation is determined by the sign
of $r$: the normals are pointing outwards if $r > 0$. If $r=0$
then $S_{\textbf{p},0}=\{\textbf{p}\}$  is a point.
Let $\textbf{n}\in \R^3$ with $\textbf{n} \cdot \textbf{n}=1$ and $h\in\R$. The oriented plane $P_{\textbf{n},h}$ in $\R^3$ is the set
\begin{equation} 
    P_{\textbf{n},h} = \{\textbf{v}\in \R^3 | \textbf{v} \cdot \textbf{n}=h \}. \nonumber
\end{equation}
The {\sl Lie scalar} product with signature $(4,2)$ in $\R^6$ is
defined by the formula
 \begin{equation} 
    [x,z]=\frac{-x_1z_2-x_2z_1}{2}+x_3z_3+x_4z_4+x_5z_5-x_6z_6.  \nonumber
\end{equation}
for $x=(x_1,\ldots,x_6)$ and $z=(z_1,\ldots,z_6)$. In matrix notation we have
\begin{equation}\label{lieproduct-matrix}
    [x,z]=xCz^T,\ \mbox{ where}\
    xC=(-x_2/2,-x_1/2,x_3,x_4,x_5,-x_6).
\end{equation}
 Denote $\yyq=(u:y_0:y_1:y_2:y_3:y_4) \in \bp(\R^6)=\bp^5$ and define the quadric
\begin{equation}\label{liequadric}
    \Qc=\{\yyq \in \bp^5 \ |\ [\yyq,\yyq]=-uy_0+y_1^2+y_2^2+y_3^2-y_4^2=0\}
\end{equation}
where $[.,.]$ is the obvious extension of the Lie scalar product to $\bp^5$.
$\Qc$ is called {\sl Lie quadric}.

We represent an oriented sphere $S_{\textbf{p},r}$ (or an oriented plane
$P_{\textbf{n},h}$) as a point $Lie(S_{\textbf{p},r})$ (resp. $Lie(P_{\textbf{n},h})$) on the Lie
quadric:
\begin{eqnarray} 
Lie(S_{\textbf{p},r})&=&(2(\textbf{p}\cdot \textbf{p} -r^2),2,2\textbf{p},2r) \in \Qc, \quad \textbf{p}\in\R^3,r\in\R , \nonumber \\
Lie(P_{\textbf{n},h})&=&(2h,0,\textbf{n},1) \in \Qc,\quad  \textbf{n}\in\R^3, \ h\in\R.   \nonumber
\end{eqnarray}

It is easy to see that we have determined a bijective
correspondence between the set of points on the Lie quadric $\Qc$ and
the set of all oriented spheres/planes in $\R^3$. Here we assume
that a point $q=(1:0:0:0:0:0) \in \Qc$ on the Lie quadric $\Qc$ corresponds
to an infinity, i.e. to a point in the compactification of $\R^3$. We
say that $q=(1:0:0:0:0:0)$ is the improper point on the Lie quadric.
Notice that  oriented planes in $\R^3$ correspond to points
$\Qc \cap T_q$, where $T_q=\{\yyq=(u:y_0:y_1:y_2:y_3:y_4) \in \bp^5 | y_0=0\}$ is
a tangent hyperplane to the Lie quadric at the improper point $q$.

Two oriented spheres $S_{\textbf{p}_1,r_1},S_{\textbf{p}_2,r_2}$ are in {\sl
oriented contact} if they are tangent and have the same
orientation at the point of contact. The analytic condition for
oriented contact is
\begin{equation} 
    \Vert \textbf{p}_1-\textbf{p}_2\Vert=|r_1-r_2|,  \nonumber
\end{equation}
where $\Vert \textbf{p}_1-\textbf{p}_2\Vert$ denotes the usual distance between two points in
the Euclidean space $\R^3$.  One can check directly that the analytical
condition of oriented contact on the Lie quadric is equivalent to
the equation
\begin{equation} 
    [Lie(S_{\textbf{p}_1,r_1}),Lie(S_{\textbf{p}_2,r_2})]=0.   \nonumber
\end{equation}

It is known that the Lie quadric contains projective lines but no
linear subspaces of higher dimension (Chapter 1, Corollary 5.2 in
\cite{cecil}). Moreover, the line in $\bp^5$ determined by two
points $k_1,k_2$ of $\Qc$ lies on $\Qc$ if and only $[k_1,k_2]=0$,
i.e. the corresponding spheres to $k_1,k_2$ are in an oriented
contact (Chapter 1, Theorem 1.5.4 in \cite{cecil}). The points on
a line on $\Qc$ form so called {\sl parabolic pencil} of spheres.
All spheres which correspond to a line on $\Qc$ are precisely the
set of all spheres in an oriented contact.

\begin{rem}
Here we use a slightly different coordinate system in Lie Geometry
than in the book \cite{cecil}.
 The scalar product as in \cite{cecil} may be obtained applying
 the following transformation:\\
$x_1'=(x_1+x_2)/2, x_2'=(x_2-x_1)/2, x_3'=x_3, x_4'=x_4, x_5'=x_5,
x_6'=x_6.$
\end{rem}

We  show now that the set of points $\yyq$ in $\Qc$ with $y_0\not= 0$ is
naturally diffeomorphic to the affine space $\R^4$. This
diffeomorphism is defined by the map
$$\begin{array}{lcc}
    \phi : \qquad \qquad \Qc \setminus T_q  & \rightarrow& \R^4,
      \nonumber \\
     (u:y_0:y_1:y_2:y_3:y_4) &\mapsto& \left( \frac{y_1}{y_0},\frac{y_2}{y_0},
    \frac{y_3}{y_0},\frac{y_4}{y_0}\right),
\end{array}$$
 where $T_q=\{\yyq=(u:y_0:y_1:y_2:y_3:y_4) \in \bp^5 \ |\ y_0=0\}$ as before, i.e. the tangent
 hyperplane to the Lie quadric $\Qc$ at the improper point
 $q=(1:0:0:0:0:0)$.
 Let $v=(v_1,v_2,v_3,v_4), w=(w_1,w_2,w_3,w_4) \in\R^4$ and  denote by
  \begin{equation}\langle v, w \rangle =
v_1 w_1+v_2 w_2+v_3 w_3 - v_4 w_4       \nonumber
\end{equation}
 the Lorentz scalar  product on
$\R^4$, which can be seen as the restriction of the Lie scalar product $[.,.]$ to $\R^4$.
The affine space $\R^4$ with the Lorentz scalar product is
called the Lorentz space and denoted by $\R^4_1$.

Let $\yy=(y_1,y_2,y_3,y_4) \in \R^4$. One can check that inverse
map of $\phi$ is given by the formula:
 \begin{equation} 
    \phi^{-1}(\yy)=(\langle \yy,\yy \rangle,
1,\yy)  \in \Qc \setminus T_q  \nonumber
\end{equation}
\textbf{ }\\
 Notice, that $\phi(Lie(S_{\textbf{p},r}))=(\textbf{p},r)$, i.e. the sphere $S_{\textbf{p},r}\in\R^3$
corresponds to a point $(\textbf{p},r)\in\R^4_1$. The map $\phi$
can be extended to a linear projection $\Phi$ from $\Qc
\setminus\{q\}$ to $\bp^4$ defined as
$$\begin{array}{lcc}\label{prj}
    \Phip : \qquad \qquad \Qc \setminus \{q\}  & \rightarrow& \bp^4
      \nonumber \\
     (u:y_0:y_1:y_2:y_3:y_4) &\mapsto& (y_0:y_1:y_2:y_3:y_4)
\end{array}$$

The points of $\Qc\cap T_q$  can be represented as
$Lie(P_{\textbf{n},h})=(2h,0,\textbf{n},1)$ and these points
correspond to planes in $\R^3$. Note that
$$\Phip(Lie(P_{\textbf{n},h}))=(0,\textbf{n},1)\in
\Omega=\{y_0=0,y_1^2+y_2^2+y_3^2-y_4^2=0\}$$ are infinite points
to the natural extension of $\R^4$ to $\bp^4$ which correspond to
a pencil of parallel planes in $\R^3$. The quadric $\Omega$ is
called \textit{absolute quadric}. The preimage of the map $\Phip$ has the
following form
\begin{equation} \label{inversep}
    \Phip^{-1}(\yyp)=(\langle y,y\rangle:y_0^2:y_0y_1:y_0y_2:y_0y_3:y_0y_4)\in \Qc  \setminus \{q\}
\end{equation}
\textbf{ }\\
 where $\yyp=(y_0:y_1:y_2:y_3:y_4)\in \bp^4$ and $\yy=(y_1,y_2,y_3,y_4)$ as before.

 A direct computation shows that for $v,w\in\R^4$
\begin{equation}\label{twoproducts}
-2[\phi^{-1}(v),\phi^{-1}(w)]=\langle v-w,v-w\rangle
\end{equation}
The formula shows that two oriented spheres defined by $v$, $w$
(i.e.  spheres $S_{(v_1,v_2,v_3),v_4}$ and
$S_{(w_1,w_2,w_3),w_4}$) are in oriented contact if and only if
 $\langle v-w, v-w \rangle = 0$.

 Let us define two maps: an embedding $i_d:\R^3\to\R^4,i_d(\textbf{p})=(\textbf{p},d),d\in\R$ and a projection
 $\pi:\R^4\to\R^3,\pi(\textbf{p},r)=\textbf{p}$, where $r\in\R$. We will
 treat points $i_0(\R^3)$ as spheres with zero radius and identify
 them with $\R^3$. All interrelations between the spaces introduced above can be described in the following diagram
\begin{equation}\label{diagram}
    \begin{array}{ccccccc}
       &  & \Qc\setminus T_q & \subset & \Qc\setminus \{q\} & \subset & \bp^5 \\
       &  & \downarrow\phi &  &\downarrow \Phip &  &  \\
      \R^3 & \stackrel{i_d}{\ \to} & \R^4 & \subset & \bp^4 &  &  \\
      \| &  &  \downarrow\pi&  &  &  &  \\
      \R^3 & = & \R^3 &  &  &  &  \\
    \end{array}
\end{equation}


\begin{defn} For an oriented surface (curve or point)
$\Mc\subset\R^3$ define an {\it isotropic hypersurface}
$\Gc(\Mc)\subset\bp^4$ as the union of all points in $\R^4$ which
correspond to oriented tangent spheres of $\Mc$. Let
$\Gc_d(\Mc)=\Gc(\Mc)\cap\{y_4=dy_0\}$ be a variety which
corresponds to tangent spheres with radius $d$ of $\Mc$. The set
$\Env_d(\Mc)=\pi(\Gc_d(\Mc)|_{\R^4})\subset\R^3$ are centers
of spheres with radius $d$ tangent to $\Mc$. The set
$\Env_d(\Mc)$ is called $d$-envelope of the variety $\Mc$.
Since $\Gc(\Mc)=\bigcup_d\Gc_d(\Mc)$ we can treat the isotropic
hypersurface $\Gc(\Mc)$ as the union of all $d$-envelope to the
variety $\Mc$.
\end{defn}

If $ y,a\in\R^4,\ a_0,y_0\in\R$, we define a function
\begin{eqnarray}\label{g}
    g((a_0:a),(y_0:y))&=&\langle ay_0-a_0y, ay_0-a_0y \rangle=
    y_0^2 a_0^2 \left< \frac{a}{a_0}-\frac{y}{y_0}, \frac{a}{a_0}-\frac{y}{y_0} \right>\\
    &=&y_0^2\langle a,a\rangle -2a_0y_0\langle
a,y\rangle + a_0^2 \langle y,y\rangle. \nonumber \
\end{eqnarray}
 Let $(y_0:y)$ be such that $g((a_0:a),
(y_0:y))=0$. By the formula (\ref{twoproducts}) we see that
spheres
$S_{\left(\frac{a_1}{a_0},\frac{a_2}{a_0},\frac{a_3}{a_0}\right),\frac{a_4}{a_0}}$
and
$S_{\left(\frac{y_1}{y_0},\frac{y_2}{y_0},\frac{y_3}{y_0}\right),\frac{y_4}{y_0}}$
are in oriented contact. Therefore, in the same manner   as
previously, we define the isotropic hypersurface $ \Gc((a_0:a))$
as follows
$$ \Gc((a_0:a))=\{(y_0,y)\in\bp^4\ |\ g( (a_0:a), (y_0:y)) = 0\}
\subset \bp^4.$$ In fact,  $ \Gc((a_0:a))$ is a quadratic cone
with a singular point at a vertex $(a_0:a)\in\bp^4$ and may be
viewed as the set of all spheres which touches the fixed sphere
$S_{\left(\frac{a_1}{a_0},\frac{a_2}{a_0},\frac{a_3}{a_0}\right),\frac{a_4}{a_0}}\
 .$
 After the restriction to the
linear subspace $y_4=dy_0$ this hypersurface consists of all
spheres with radius $d$ which are in oriented contact with the
sphere
$S_{\left(\frac{a_1}{a_0},\frac{a_2}{a_0},\frac{a_3}{a_0}\right),\frac{a_4}{a_0}}$
which we denote  as
 $\Gc_d((a_0:a))=\Gc((a_0:a))\cap\{y_4=dy_0\}.$ We notice that
 $\Gc_d((a_0:a))|_{y_0=1}$ is defined by the equation
  $ (a_1-a_0y_1)^2+(a_2-a_0y_2)^2+(a_3-a_0y_3)^2=(a_4-a_0d)^2$,
 i.e.
 \begin{eqnarray*}\label{gg}
\pi(\Gc_d((a_0:a))|_{\R^4})&=&S_{\left(\frac{a_1}{a_0},\frac{a_2}{a_0},\frac{a_3}{a_0}\right),\frac{a_4-a_0d}{a_0}}=
\text{Env}_{-d}\left(S_{\left(\frac{a_1}{a_0},\frac{a_2}{a_0},\frac{a_3}{a_0}\right),\frac{a_4}{a_0}}\right)
\quad\text{and}\\
\Gc_d((a_0:a))|_{\R^4}&=&i_d\left(S_{\left(\frac{a_1}{a_0},\frac{a_2}{a_0},\frac{a_3}{a_0}\right),\frac{a_4-a_0d}{a_0}}\right)
\end{eqnarray*}
Therefore, in this case, the isotropic hypersurface $\Gc((a_0:a))$
may be treated as a union  all envelopes to the sphere
$S_{\left(\frac{a_1}{a_0},\frac{a_2}{a_0},\frac{a_3}{a_0}\right),\frac{a_4}{a_0}}\
.$  In the next section we generalize the definition of the
isotropic hypersurface $\Gc(\Mc)$ for a curve $\Mc$ in $\R^4$
(or  $\bp^4$).

All lines in $\R^4_1$ with directional vectors $v$ can be
classified into three types depending on the sign of $ \langle v,v
\rangle$: $(+)$-lines, $(0)$-lines (also called \emph{isotropic}
lines), and $(-)$-lines.

\section{The isotropic hypersurface and $d$-envelopes}\label{canal}

In this section, we will see that the definition of the canal surface is not obvious and
and we will introduce some geometrical object related to it.
A canal surface is given by a so-called \textit{spine curve} $\Ec$, which is the closed image (with respect to the
Zariski topology) of a rational map
$$\begin{array}{lcc}\label{eee}
     \R  & \dashrightarrow& \R^4
      \nonumber \\
     t &\mapsto& \left(\frac{e_1(t)}{e_0(t)},\frac{e_2(t)}{e_0(t)},\frac{e_3(t)}{e_0(t)},\frac{e_4(t)}{e_0(t)}\right)
\end{array}$$
with polynomials $e_0(t),\ldots,e_4(t) \in \R[t]$ such that $n=\max_{i=0,.,4}\{\deg(e_i(t))\}$. For
abbreviation, we usually skip the variable $t$ in the notations.
The spine curve describes a family of spheres
$\{S_{\left(\frac{e_1(t)}{e_0(t)},\frac{e_2(t)}{e_0(t)},\frac{e_3(t)}{e_0(t)} \right),\frac{e_4(t)}{e_0(t)}}
\ | \ t \in \R \}$ whose centers are given by the first three coordinates
$\left(\frac{e_1(t)}{e_0(t)},\frac{e_2(t)}{e_0(t)},\frac{e_3(t)}{e_0(t)} \right)$
and whose radii are given by the last coordinate~$\frac{e_4(t)}{e_0(t)}$. Intuitively, the canal surface is the envelope
of this family of spheres, but there are some subtleties to consider before we can make a precise definition.

We can also consider the spine curve as a projective curve $\Ep$ given as the closed image of a
parametrization
$$\begin{array}{lcc}
     \bp^1  & \dashrightarrow& \bp^4
      \nonumber \\
     t &\mapsto& (e_0(t):e_1(t):e_2(t):e_3(t):e_4(t))
\end{array}$$
with the non-restrictive condition $\gcd(e_0,\ldots,e_4)=1$, which
means that there are no base-points (i.e. parameters for which the
map is not well-defined).

Note that in this case the polynomials $e_i$ are actually to be
considered as homogenized to the same degree $n$ with respect to a
new variable $s$. As there is a one-to-one correspondence between
the univariate polynomials of a certain degree and their
homogeneous counterparts, we will keep the notation from above and
distinguish between the affine and projective case only where it
is necessary to avoid confusion.

In the following we use the notations
$$e=(e_1,e_2,e_3,e_4),\ \yy=(y_1,y_2,y_3,y_4),$$
$$\eep=(e_0:e_1:e_2:e_3:e_4),\ \yyp=(y_0:y_1:y_2:y_3:y_4).$$
 We first
proceed to define a hypersurface in $\bp^4$ which is closely
related to the canal surface.

\begin{defn}
The {\it isotropic hypersurface} $\Gc(\Ep)=\{\yyp\ |\ G(\yyp)=0 \}
\subset \bp^4$ associated with the (projective) spine curve $\Ep$
is the variety in $\bp^4$ defined by the polynomial
$\G(\yyp)=\Res{}_t(g_1,g_2)$ where
\begin{eqnarray*} 
g_1(\yyp,t) & = &g(\eep,\yyp)= (e_0y_1-e_1y_0)^2+(e_0y_2-e_2y_0)^2+\\
& & \qquad \qquad (e_0y_3-e_3y_0)^2-(e_0y_4-e_4y_0)^2 \\
            &=& \langle e_0y-y_0e, e_0y-y_0e \rangle  =  e_0^2 \langle y,y\rangle -2\langle
e_0e,y_0y\rangle +y_0^2\langle e,e\rangle, \\
g_2(\yyp,t) &= &\frac{\partial
 g_1(\yyp,t)}{\partial t} =
 2(e_0e_0'\langle y,y\rangle-\langle (e_0e)',y_0y\rangle + y_0^2\langle e',e\rangle).
\end{eqnarray*}
So, we define  $\Gc(\Ep)$ as the envelope of the family of
isotropic hypersurfaces $\Gc(\eep)=\Gc((e_0(t):e(t))$.

\end{defn}

In the previous section we showed that
$\Gc_d(\eep)|_{\R^4}=\Env_{-d}\left(S_{\left(\frac{e_1}{e_0},\frac{e_2}{e_0},\frac{e_3}{e_0}\right),\frac{e_4}{e_0}}\right)$.
This interpretation leads to the following definition.

\begin{defn}
The \textit{$d$-envelope} associated with the (projective) spine curve $\Ep$ is defined
as the hypersurface $\Env_d(\Ep) \subset \bp^3$
 given by the implicit equation
 $$\G_d(y_0,y_1,y_2,y_3)=\Res{}_t(g_1|_{y_4=-dy_0} ,g_2|_{y_4=-dy_0})=\Res{}_t(g_1,g_2)|_{y_4=-dy_0},$$
i.e. the equation obtained by replacing $y_4$ in $\G(\yyp)$ by $-dy_0$, where $d \in \R$.

The \textit{affine envelope} $\Env_d(\Ec)$ at distance $d$ is the
restriction of $\Env_d(\Ep)$ to the affine space $\R^3$, defined
by the equation $\G_d |_{y_0=1}=\Res{}_t(g_1,g_2)|_{y_4=-dy_0,
y_0=1}$, i.e. by setting $y_0=1$.
\end{defn}

So $\Gc(\Ep)$ contains all offsets associated with the spine curve $\Ep$. Indeed, the surface 
\begin{equation} 
{\Env}_{d}(\Ep) = \Gc(\Ep) \cap \{ y_4 = -dy_0 \}  \nonumber
\end{equation}
is a hyperplane section of $\Gc(\Ep)$, which can be interpreted as a parametrization
of all offsets (with respect to the parameter $y_4$).

The special case $d=0$ is particularly important.
For the real part of $\Env_0(\Ec)$ to be non-empty, one has to suppose that
$\Ec$ has tangent $(+)$-lines almost every\-where, or equivalently that $\langle e,e \rangle > 0$ almost
everywhere. $\Env_0(\Ec)$ is the
envelope of the family of spheres in $\R^3$ given by the spine curve $\Ec$ and
$\Env_d(\Ec)$ is the envelope of the same family of spheres with radii augmented by $d$. For instance,
circular cylinders or circular cones (call them just \emph{cones}) are envelopes
$\Env_0(\mathcal{L})$ of $(+)$-lines $\mathcal{L}$ and vice versa.
In the literature, the canal surface $\Cc$ is usually defined as this envelope $\Env_0(\Ec)$.
However, we will show in an example that these envelopes can contain ``unwanted'' extraneous factors, which are geometrically counterintuitive.

\begin{exmp}
Consider the spine curve $\Ec$ given by
$$\left(\frac{e_1(t)}{e_0(t)},\frac{e_2(t)}{e_0(t)},\frac{e_3(t)}{e_0(t)},\frac{e_4(t)}{e_0(t)}\right)=
\left(\frac{1-t^2}{1+t^2},\frac{2t}{1+t^2},0,\frac{1}{2}\right).$$
The first three coordinates describe a circle in the plane and moving spheres of constant radius along this curve,
so intuitively the
envelope should be a torus $\mathcal{T}$. But it turns out that the implicit equation of $\Env_0(\Ec)$
is up to a constant computed as
$$\G_0=\Res{}_t(g_1,g_2)|_{y_4=0, y_0=1}=(y_1^2+y_2^2)^2(4y_1^2+4y_2^2+4y_3^2+8y_1+3) F_\mathcal{T}$$
where $F_\mathcal{T}$ is indeed the equation of the torus. To
understand where the other factors come from, consider the
following: For a given parameter $t$, the equations $g_1$ and
$g_2$ define spheres $\mathcal{S}_1(t)$ and $\mathcal{S}_2(t)$ in
$\R^3$ and $$\Env{}_0(\Ec)=  \bigcup_t \mathcal{S}_1(t) \cap
\mathcal{S}_2(t)$$ of the intersections of these spheres (actually
this is nothing else than the geometric definition of the
resultant). Now, while for almost all $t$ this intersection is a
transversal circle on the torus (often called characteristic
circle in the literature), it can happen that the spheres
degenerate either to planes or to the whole space. In our example,
for the parameters $t=i$ and $t=-i$ we have $g_1(i)=g_1(-i)=0$,
$g_2(i)=-iy_1+y_2$ and $g_2(-i)=iy_1+y_2$, so the intersection in
those parameters actually degenerates to (complex) planes which
correspond to the factor $(-iy_1+y_2)(iy_1+y_2)=y_1^2+y_2^2$. In
the parameter $t=\infty$, both $g_1$ and $g_2$ define the same
sphere whose equation $4y_1^2+4y_2^2+4y_3^2+8y_1+3$ is the other
extraneous factor. This kind of phenomenon can also happen for
real parameter values, but it is interesting to remark that even
though we consider a real parametrization, non-real parameters can
interfere with the envelope, because the resultant ``knows'' about
them.

This example shows that $\Env_0(\Ec)$ is not a suitable definition for the canal surface
$\Cc$ and we will later develop one that avoids the
kind of extraneous components we have observed.
\end{exmp}

\begin{rem} Sometimes, in the literature, $\Env_0(\Ec)$ is defined
in affine space as  the resultant
\begin{eqnarray*}
    \check{\G}_0(y_1,y_2,y_3)&=&\Res{}_t(\check{g}_1(y_1,{y_2},y_3,t),\check{g}_2(y_1,y_2,y_3,t)),
    \ \mbox{where}\\ \check{g}_1&=&e_0^2\check{f}_1,\quad\quad \check{g}_2=e_0^3\check{f}_2,\\
\check{f}_1&=&\left(y_1-\frac{e_1}{e_0}\right)^2+\left(y_2-\frac{e_2}{e_0}\right)^2
+\left(y_3-\frac{e_3}{e_0}\right)^2-\left(\frac{e_4}{e_0}\right)^2,\\
\check{f}_2&=&\frac{\partial\hat{f}_1}{\partial t}
\end{eqnarray*}
or in other words by deriving the affine equation of the sphere
after the substitutions $y_4=0,y_0=1$ and homogenizing afterwards.
Note that in this case $\check{f}_2$ and $\check{g}_2$ are linear
in $y_1,y_2,y_3$. Let $\tilde{f}_1=g_1|_{y_4=0, y_0=1}$ and
$\tilde{f}_2=g_2|_{y_4=0, y_0=1}$. An easy computation shows that
we have the following equalities
\begin{equation*}
    \check{g}_1=\tilde{f}_1,\quad
    \check{g}_2=e_0\tilde{f}_2-2e_0'\tilde{f}_1.
\end{equation*}
Therefore, by  properties of the resultant
(\ref{res-product}),(\ref{res-divisor}), we have
\begin{eqnarray*}
\Res{}_t(\check{g}_1,\check{g}_2)&=&\Res{}_t(\tilde{f}_1,e_0\tilde{f}_2-2e_0'\tilde{f}_1)\\
&=& \Res{}_t(\tilde{f}_1,e_0\tilde{f}_2)\\
&=&\Res{}_t(\tilde{f}_1,e_0)\cdot
\Res{}_t(\tilde{f}_1,\tilde{f}_2).
\end{eqnarray*}
Hence, we have $\check{\G}_0=\Res{}_t(\tilde{f}_1,e_0)\cdot
{\G}_0$, so there are even more extraneous factors than before due to the
roots of $e_0$.
\end{rem}

\subsection*{Linearizing the problem}

The main idea to understand and eliminate the extraneous
components that appeared in the example is to linearize the
equations $g_1$ and $g_2$ by replacing the quadratic term $\langle
\yy,\yy \rangle$ by a new variable $u$ (or more precisely $uy_0$
to keep the equations homogeneous). This will make the results
developed in Section 2 applicable. Geometrically, this means that
we will pull back the spine curve to $\Qc$ via the correspondence
$\Phip$.

For a spine curve $\Ec \in\R^4_1$ we define a {\sl proper} pre-image
$\Eq$  in the Lie quadric $\Qc$ as the closure of the set
$\Eq=\Phi^{-1}(\Ec)$ in $\Qc$.
It is immediate by \eqref{inversep} that the parametrization of $\Eq$ is
\begin{equation} \label{bar_e}
 \begin{array}{lcc}
     \bp^1  & \dashrightarrow& \Qc\subset\bp^5
       \\
     t &\mapsto& (\langle
e(t),e(t) \rangle:e_0^2(t):e_0(t)e_1(t):e_0(t)e_2(t):e_0(t)e_3(t):e_0(t)e_4(t))
\end{array}
\end{equation}
We can now define the envelopes associated with this new spine curve as follows.
\begin{defn}
The variety $\Hc(\Eq) \subset \bp^5$ associated with $\Eq$
is the hypersurface in $\bp^5$ defined by the implicit equation $\Hh(\yyq)=\Res{}_t(h_1,h_2)$ where
$\yyq=(u:y_0:y_1:y_2:y_3:y_4)$ and
\begin{eqnarray} 
\label{Hsystem}
    h_1(\yyq,t)&=&-2[\yyq,\Eq(t)]=u
e_0^2+y_0\langle e,e\rangle-2\langle
e_0e,y\rangle,  \nonumber \\
    h_2(\yyq,t)&=& \frac{\partial h_1(\yyq,t)}{\partial t} =-2[\yyq,
    {\Eq}'(t)]=2(u
e_0e_0'+y_0\langle e',e\rangle-\langle
(e_0e)',y\rangle). \nonumber
\end{eqnarray}

Similarly, the variety $\Hc_d(\Eq) \subset \bp^4$ is defined by the implicit equation $$\Hh_d(u,y_0,y_1,y_2,y_3)=\Res{}_t(h_1|_{y_4=-dy_0} ,h_2|_{y_4=-dy_0})=\Res{}_t(h_1,h_2)|_{y_4=-dy_0},$$
i.e. the equation obtained by replacing $y_4$ in $\Hh(\yyp)$ by $-dy_0$, where $d \in \R$.
\end{defn}

Of course this is nothing else than substituting $\langle y,y \rangle$
in $g_1$ and $g_2$ by $uy_0$ and dividing by $y_0$, so $g_i(\yyp)=h_i(\langle
y,y\rangle,y_0^2,y_0y),i=1,2$, i.e. $g_i=h_i\circ
\Phi^{-1},i=1,2$. Now as an immediate corollary we obtain

\begin{prop} \label{Phi}With the notations as above we have 
\begin{equation}
\G(\yyp)=\Hh(\langle y,y\rangle,y_0^2,y_0y),\  \mbox{i.e.}\ \G=\Hh \circ
\Phi^{-1},
\end{equation} and
\begin{equation}
\G_d(y_0,y_1,y_2,y_3)=\Hh_d(y_1^2+y_2^2+y_3^2-d^2y_0^2,y_0^2,y_0y_1,y_0y_2,y_0y_3).
\end{equation}
\end{prop}

To sum up, we have defined two hypersurfaces as resultants of two quadratic forms: $\Env_d(\Ep) \subset \bp^3$, which are the offsets to the
spine curve $\Ep$, and $\Gc(\Ep) \subset \bp^4$, which can be interpreted as a parametrization of those offsets.
As seen in an example, these definitions can lead to additional components which are against the geometric intuition, so it is desirable
to give another definition which avoids those extra factors. To this end, we have linearized the problem by replacing the quadratic polynomials $g_1$ and
$g_2$ by linear forms $h_1$ and $h_2$ by substituting the quadratic term by a new variable and have seen how to reverse this substitution.
Geometrically, this means that we replace the hypersurfaces $\Env_d(\Ep)$ and $\Gc(\Ep)$ by hypersurfaces $\Hc_d(\Eq)$ and $\Hc(\Eq)$ in one dimension higher.

This has the advantage that we can now apply the technique of $\mu$-bases developed earlier to understand and eliminate the extraneous factors
of $\Hc_d(\Eq)$ and $\Hc(\Eq)$ and then come back to $\bp^3$ (resp. $\bp^4$) with the substitution formulae of Proposition \ref{Phi}.

\section{The dual variety, offsets, and the canal surface.} \label{dualvariety}

In this section, we will finally be able to define the canal surface $\Cc$ (and more general
offsets to it) and the so-called dual variety $\Gamma(\Ep)$, which can be seen as a parametrization
of the offsets to $\Cc$.

Up to the constant $-2$ the system $h_1=h_2=0$ is equal to
 \begin{equation}\label{gamma_lie}
    \left\{
    \begin{array}{ccl}
      \left[\yyq,\Eq(t)\right] =\Eq(t)C \yyq^T =0, \\
      \left[\yyq,\Eq'(t)\right] =\Eq'(t)C \yyq^T =0,  \\
    \end{array}
\right.
\end{equation}
where the matrix $C$ is defined by the formula
\eqref{lieproduct-matrix}.

We can interpret the variety $\Hc(\Eq)$ defined by \eqref{gamma_lie} as a dual variety to
the curve $\Eq$ with respect to the Lie
quadric $\Qc$, i.e. the dual variety to the curve $\Eq(t)C$. Indeed, this dual variety
consists of the hyperplanes which touch the curve $\Eq(t)C$. The
first equation in (\ref{gamma_lie}) means that the hyperplane
contains the point $\Eq(t)C$, the second equation means that the
hyperplane contains the tangent vector
$\Eq'(t)C$ to the curve $\Eq(t)C$. 

In order to simplify notation we denote

\begin{eqnarray}\label{EEprime}
\E &= &\Eq(t)C =\left(-\frac{e_0^2}{2},-\frac{\langle
e,e\rangle}{2},
e_0e_1,e_0e_2,e_0e_3,-e_0e_4\right)    \\
\E' &= &\Eq'(t)C =(-e_0e'_0,-\langle e',e\rangle,
e'_0e_1+e_0e'_1,e'_0e_2+e_0e'_2,\nonumber \\
 & & \qquad \qquad \: \: e'_0e_3+e_0e'_3,-e'_0e_4-e_0e'_4 )
\end{eqnarray}
and we have that $\Hc(\Eq)=\Sc_{\E,\E'}$ by \eqref{druledAB}. As
we have seen in Section 2, this surface contains extraneous
factors which correspond to the roots of the $2$-minors of the
matrix $W_{\E,\E'}$, but which can be eliminated by replacing
$\E,\E'$ by a $\mu$-basis of the module $\langle \E,\E'\rangle$.
It is thus natural to make the following definition.

\begin{defn}
 We define the dual variety $\Vc(\Eq) \subset \bp^5$ to the curve $\Eq$ as the  hypersurface
\begin{equation}\label{vde}
    \Vc(\Eq) = \Sc_{\langle \E,\E'\rangle}
\end{equation}
where $\langle \E,\E'\rangle$ is the module quasi-generated by $\E$ and $\E'$.
\end{defn}

By the results of Section \ref{quasigens}, it is immediate that $\Vc(\Eq) \subset \Hc(\Eq)$ does not
contain the components of $\Hc(\Eq)$ caused by parameters $t$ where $W_{\E,\E'}(t)$ is not of full
rank or equivalently, where the intersection of the hyperplanes defined by $h_1$ and $h_2$ is of
codimension 1, i.e. the hyperplanes coincide. So we can deduce

\begin{prop}\label{impleq}
Let $\E_1,\E_2$ be a $\mu$-basis of the module quasi-generated by
$\E$ and $\E'$ and let $k$ be the degree of the parametrization $E \wedge E'$ as in Section \ref{quasigens}. Then
$$k \cdot \deg \Vc(\Eq)=\deg \E_1+\deg \E_2=\deg (\E \wedge \E')-\deg q_{\E,\E'},$$ where $q_{\E,\E'}=\gcd(\E \wedge
\E')$ and $$\Res{}_t(\E_1 \cdot \yyq^T, \E_2 \cdot \yyq^T)=F_{\Vc(\Eq)}^k,$$ where $F_{\Vc(\Eq)}$ is the implicit equation of
$\Vc(\Eq)$.
\end{prop}
\begin{proof} It follows directly from Propositions~\ref{degreeofSM} and~\ref{s_m}.
\end{proof}

Of course, the  same considerations can be applied to the
hypersurfaces $\Hc_d(\Eq)$ and we make the analogous definitions.
Substituting $y_4=-dy_0$ in $h_1$ and $h_2$ corresponds to
replacing $\E$ and $\E'$ by two linear forms
\begin{eqnarray}\label{DDprime}
D &= &\left(-\frac{e_0^2}{2},-\frac{\langle
e,e\rangle}{2}+de_0e_4,
e_0e_1,e_0e_2,e_0e_3\right)    \\
D' &= & \left(-e_0e'_0,-\langle
e',e\rangle+d(e'_0e_4-e_0e'_4),  e'_0e_1+e_0e'_1,e'_0e_2+e_0e'_2,e'_0e_3+e_0e'_3\right) \nonumber
\end{eqnarray}
with $D,D' \in \R^5$. Now $\Hc_d(\Eq)=\Sc_{D,D'}$ and one makes an analogous definition:

\begin{defn}
 We define the hypersurface $\Vc_d(\Eq)$ as
\begin{equation}
    \Vc_d(\Eq)= \Sc_{\langle D,D'\rangle}\subset  \bp^4
\end{equation}
where $\langle D,D'\rangle$ is the module quasi-generated by $D$ and $D'$.
\end{defn}

In this case also, $\Vc_d(\Eq) \subset \Hc_d(\Eq)$ does not contain extraneous factors due
to the parameters $t$ where the rank of $W_{D,D'}(t)$ drops. At this point, it should be remarked
that while we clearly always have
\begin{equation} 
\Vc_{d}(\Eq) \subset \Vc(\Eq) \cap \{ y_4 = -dy_0 \}  \nonumber
\end{equation}
this inclusion is not necessarily an equality (note that we had ${\Env}_{d}(\Ep) = \Gc(\Ep) \cap \{ y_4 = -dy_0 \}$ for
the corresponding varieties). Analogously to Proposition \ref{impleq} the following holds.

\begin{prop}\label{impleq1}
Let $D_1,D_2$ be a $\mu$-basis of the module quasi-generated by $D$ and $D'$ and let $k$ be the degree of the parametrization $D \wedge D'$ as in Section \ref{quasigens}. Then
$$\deg \Vc_d(\Eq)=\deg D_1+\deg D_2=\deg (D \wedge D')-\deg q_{D,D'},$$ where $q_{D,D'}=\gcd(D \wedge
D')$ and $$\Res{}_t(D_1 \cdot (u,y_0,y_1,y_2,y_3)^T, D_2 \cdot (u,y_0,y_1,y_2,y_3)^T)=F_{\Vc_d(\Eq)}^k$$ where $F_{\Vc_d(\Eq)}$ is the implicit equation of
$\Vc_d(\Eq)$.
\end{prop}
\begin{proof} It follows directly from Propositions~\ref{degreeofSM} and~\ref{s_m}.
\end{proof}

Finally, we can use the correspondance of Proposition~\ref{Phi} to define the canal surface.

\begin{defn}\label{Gamma-definition}
The $\Gamma$-hypersurface is defined as
 $$\Gamma(\Ep)=\Phi(\Vc(\Eq) \cap \Qc),$$
and the offset $\Off_d(\Ep)$ at distance d to the canal surface $\Cc$ is
 $$\Off{}_d(\Ep)=\Phi_0(\Vc_d(\Eq) \cap \Qc_d),$$
where $\Qc_d=\{ (u:y_0:y_1:y_2:y_3) \in \bp^4 \ |\ -uy_0+y_1^2+y_2^2+y_3^2-d^2y_0^2=0\}$ and
$\Phi_0(u,y_0,y_1,y_2,y_3)=(y_0,y_1,y_2,y_3)$. The canal surface itself is the
special case $d=0$ or in other words $\Cc=\Off_0(\Ep)$.
\end{defn}

Note that the extraneous factors of $\Hc_d(\Eq)$ and $\Hc(\Eq)$
are in one-to-one correspondence with the extraneous factors of
the corresponding hypersurfaces $\Env_d(\Ep)$ and $\Gc(\Ep)$ since
they are caused by parameter values where the intersection of
$h_1$ and $h_2$ (resp. $g_1$ and $g_2$) is of codimension one.
So $\Gamma(\Ep)$ and $\Cc_d(\Ep)$ contain no such factors.\\

In this section and the previous one, many different geometric objects
have been defined. We illustrate in the following diagram how they are
related in order to make the situation clearer.

\begin{equation}\label{diagram2}
    \begin{array}{ccccccc}
    \bp^4\        & &  \bp^5 & &  \bp^5 & & \bp^4  \\
    \cup\        & &  \cup & &  \cup & & \cup  \\
        \Gc(\Ep) & \stackrel{\Phip}{\longleftarrow} & \Hc(\Eq)\cap\Qc & \supseteq & \Vc(\Eq)\cap\Qc & \stackrel{\Phip}{\longrightarrow} & \Gamma(\Ep) \\
\cup\        & &  \cup & &  \cup & & \cup  \\
       \Env{}_d(\Ep) & \stackrel{\Phip_d}{\longleftarrow} & \Hc_d(\Eq)\cap\Qc_d & \supseteq & \Vc_d(\Eq)\cap\Qc_d & \stackrel{\Phip_d}{\longrightarrow} & \Off{}_d(\Ep)
\\
\cap        & &  \cap & &  \cap & & \cap  \\
 \bp^3        & &  \bp^4 & &  \bp^4 & & \bp^3\\
    \end{array}
\end{equation}

Note that the hypersurfaces in the third row are included in the
corresponding hypersurfaces in the second row. The first column is
the naive definition of the objects to be studied: $\Env{}_d(\Ep)$
is more or less a $d$-offset to the canal offsets and $\Gc(\Ep)$ a
hypersurface in one dimension higher containing all those offsets.
However, they contain extraneous factors. So by passing to the
second column, we linearize the hypersurfaces (i.e. we express
them as resultants of linear forms) and can apply $\mu$-bases to
eliminate the extraneous factor, which gives the third column and
finally go back down in dimension (by intersecting with $\Qc$ and
applying $\Phi$ to obtain the objects we are interested in: the
offsets $\Off_d(\Ep)$ (in particular the canal surface
$\Cc=\Off_0(\Ep)$) and the $\Gamma$-hypersurface.

\subsection{The implicit equation.}

 We can now describe how to compute powers of the implicit equations of the dual varieties
 $\Vc(\Eq)$ and $\Vc_d(\Eq)$, the hypersurface $\Gamma(\Ep)$ and the offsets
 surface $\Cc_d(\Ep)$. We should remark that these powers (which are the degrees of the parametrizations
of the corresponding Pl\"ucker curves) are in a way inherent to the geometry of the problem, as
we shall illustrate in Example~\ref{doubleellipsoid}. They can be interpreted as the number of times
the surface is traced by the spine curve. Note also that this not necessarily due to the non-properness
of the spine curve: Even for a proper spine curve it can happen that the canal surface
(or its offsets) is multiply traced, as in Example~\ref{doubleellipsoid}.\\

\textbf{Algorithm (implicit equations)}

\begin{enumerate}
 \item INPUT: A rational vector $e(t) \in \mathbb{R}(t)^4$ as in formula (\ref{eee}).

 \item Define $\E,\E' \in\R[t]^6$ as in formula \eqref{EEprime} and $D,D' \in \R[t]^5$ as
in formula \eqref{DDprime}.

 \item Compute a $\mu$-basis $\E_1,\E_2$ of the module  $\langle \E,\E'\rangle$ and a $\mu$-basis $D_1,D_2$ of the
module $\langle D,D'\rangle$ using the algorithm in Section \ref{quasigens}.

 \item Set $F_{\Vc(\Eq)} =\Res{}_t(\E_1 \cdot \yyq^T, \E_2 \cdot \yyq^T)$ and
$F_{\Vc_d(\Eq)} =\Res{}_t(D_1 \cdot (u,y_0,y_1,y_2,y_3)^T, \\D_2 \cdot (u,y_0,y_1,y_2,y_3)^T)=0$.

\item Let $F_{\Gamma(\Ep)}(y_0,y_1,y_2,y_3,y_4)=y_0^k
F_{\Vc(\Eq)}((y_1^2+y_2^2+y_3^2-y_4^2)/ y_0,y_0,y_1,y_2,y_3,y_4)$,
where $k$ is a minimal integer such that $F_{\Gamma(\Ep)}$ is a
polynomial. Similarly, set\\ $
F_{\Cc_d(\Ep)}(y_0,y_1,y_2,y_3)=y_0^k
F_{\Vc_d(\Eq)}((y_1^2+y_2^2+y_3^2-d^2y_0^2)/
y_0,y_0,y_1,y_2,y_3)$.

\item OUTPUT: $F_{\Vc(\Eq)}$, $F_{\Vc_d(\Eq)}$,
$F_{\Gamma(\Ep)}$, and $F_{\Cc_d(\Ep)}$, which are powers of the implicit equation of the varieties
$\Vc(\Eq)$, $\Vc_d(\Eq)$, $\Gamma(\Ep)$ and $\Cc_d(\Ep)$
\end{enumerate}

\textbf{}\\
Note that the affine parts of these equations can be obtained by
replacing $y_0=1$ before the resultant computation.

\subsection{The parametrization of the dual variety.}

 We can describe the parametrization of $\Vc(\Eq)$. The
hyperplane defined by the equation
\begin{equation} \label{tang}
\det (\yyq,\Eq(t)C,\Eq'(t)C,a_1,a_2,a_3)=A_1u+A_2y_0+A_3y_1+...+A_6y_4=0
\end{equation}
 is tangent to the curve $\Eq(t)C$ , $a_i\in \R^6, i=1,2,3$ are three
arbitrary points. By the definition a point on the dual variety
$\V$ is $(A_1,...,A_6)$. Define
$D=(\Eq(t)C,\Eq'(t)C,a_1,a_2,a_3)$ to be the $5\times 6$ matrix with
five rows $\Eq(t)C,\Eq'(t)C, a_1, a_2, \\a_3$. And let $D_i,i=1,...,6$ be
$5\times 5$ matrices obtained from $D$ by removing the i-th
column. Then using the Laplacian expansion by minors for the first
row of the determinant (\ref{tang}) we obtain  the parametrization
of $\Vc(\Eq)$  as follows:
\begin{equation}
c(D)=(\det D_1,-\det D_2,\det D_3,-\det D_4,\det D_5,-\det
D_6)/m\subset \V,
\end{equation}
where $m=\gcd(D_1,...,D_6)$. Here $t,a_1,a_2,a_3$ are arbitrary
parameters.


\section{ The implicit degree of the hypersurface $\Gamma(\Ep)$.}

The aim of this section is to get some formula for the implicit
degree of the hypersurface $\Gamma(\Ep)$ in terms of the rational
spine curve $\Ep=\{\eep(t)\in\bp^4\}$. Notice that the implicit
degree of the canal surface $\Cc$ is less or equal than
$\deg\Gamma(\Ep)$ because we always have the inclusion
$$ \Off{}_d(\Ep)\subset\Gamma(\Ep)\cap\{y_4=-dy_0\},\  \text{i.e.}\  \deg\Off{}_d(\Ep)\leq\deg\Gamma(\Ep).$$
So this formula gives upper bound for the degree of the canal
surface. In the case of a polynomial spine curve the upper bound was
obtained in the paper \cite{xu}. Note that for the computation of
the implicit degree we do not need the implicit equation of the
hypersurface. We believe that this formula is useful for higher
degree spine curves because the computation of the implicit
equation may be very difficult in practice.

Let us remind that the pre-image $\Phip^{-1}(\Gamma(\Ep)) \subset
\Qc$ is defined by the intersection of two varieties $\Vc(\Eq
)\cap \Qc$. Let denote by $G(u,y_0,y_1,y_2,y_3,y_4)=0$ the
equation of $\Vc(\Eq)$. The Lie quadric has the equation
$uy_0=\langle y,y \rangle$ (recall that $\langle y,y
\rangle=y_1^2+y_2^2+y_3^2-y_4^2$). By the definition
(\ref{Gamma-definition}) the equation of $\Gamma(\ee)$ is obtained
after the elimination of the variable $u$ from the equations of
$\Qc$ and $\Vc(\Eq)$, i.e.
\begin{equation}\label{geq}
    \Gamma(\Ep):\left\{F(y_0,y_1,y_2,y_3,y_4)=y_0^kG\left(\frac{\langle y,y
    \rangle}{y_0},y_0,y_1,y_2,y_3,y_4\right)=0\right\},
\end{equation}
where $k$ is a minimal integer such that  the left side of the
equation (\ref{geq}) is polynomial.
  We introduce the
following weighted degree
\begin{eqnarray}\label{dw}
    d_w(u^{k_1}y_0^{k_2}y_1^{k_3}y_2^{k_4}y_3^{k_5}y_4^{k_6})&=&2k_1+k_3+k_4+k_5+k_6,\\
    d_w(G(u,y_0,y_1,y_2,y_3,y_4))&=&\max_{ i }\{d_w(m_i)\}  \nonumber
\end{eqnarray}
where $G=\sum c_im_i$ is a linear combination of the monomials
$m_i$. Using this notation we have  $\deg \Gamma(\Ep)=d_w(G)$.

Let us assume that  the  curve $\Ec=\{e(t,s)\in\R^4\ |\
(t,s)\in\bp^1\}$ has a homogeneous parametrization.  We introduce
the following notations
\begin{eqnarray}
 w&=&(w_1,w_2,w_3,w_4),\ w_j=e_j'e_0-e_je_0',\
    j=1,2,3,4,\quad  \\
\gamma &=& \max_j \{\deg (w_j,t)\}  ,\label{seq-w}
\end{eqnarray}
 where  $e_i'$ means derivative with respect to $t$.
We will say that the curve $\Ep\in\bp^4$ is of  {\it general} type
if :
 \begin{eqnarray}\label{assump}
\gcd(w_1,w_2,w_3,w_4)=\gcd(e_0,e_0')=\gcd(e_0,\langle
e,e\rangle)=1,\ \deg\Ep=\deg(e_0,t)\ \mbox{and}
\nonumber\\
\mbox{the parametrization degree of the \pl curve}\
 \varphi_{\PP}:t\to E\wedge E'\  \mbox{is  one},\hphantom{eee}
\end{eqnarray}
(i.e. $\deg\varphi_{\PP}=1$), where $e=(e_1,e_2,e_3,e_4)$ and $E$
as in formula (\ref{EEprime}).

 The   first equation in the system
(\ref{gamma_lie}) has the following form:
\begin{equation}\label{h1-explicit}
h_1=\Eq C\hat{y}^T=E\hat{y}^T=-e_0^2u/2-\langle e,e \rangle
y_0/2+e_0\langle
 e,y \rangle. \
 \end{equation}  The polynomial
$h_1$ is linear in the variables $u,y_0,y_1,y_2,y_3,y_4$ and has
 degree $2n$ in the variable $t$.
The elimination of the variable $t$ from the system $h_1=h_1'=0$
is the reducible polynomial $\Res_t(h_1,h_1')=H_1...H_k G$. By
definition one of those factors is the equation of the dual
variety $\Vc(\Eq)=\{G\} $.

\begin{prop}\label{degree_of_G} If $\Ep$ is a curve of general type then
$\deg\Vc(\Eq)=4n-2$, where $n=\deg\Ep$. Moreover, we have $G\cdot
LC(h_1)=\Res_t(h_1,h_1')$, where $LC(h_1)$ is the leading
coefficient of the polynomial $h_1$ with respect to the variable
$t$ and $\{\hat{y}\in\bp^5\ |\ G(\hat{y})=0\}=\Vc(\Eq)$.
\end{prop}
\begin{proof}
For the curve of general type by the Proposition~\ref{impleq} we
have \\
$\deg\Vc(\Eq)=\deg E\wedge E'-\deg(\gcd (E\wedge E')).$ We can
compute components of the \pl vector $E\wedge
E'=([1,2]:[1,3]:...:[5,6])\in\bp^{14}$. For example,
$[1,2+j]=e_0^2w_j/2, \ j=1,2,3,4$ and $[2,3]=(-\langle e,e\rangle
(e_0e_1)'+e_0e_1\langle e,e\rangle')/2$. By assumption
(\ref{assump}) we see \\
$\gcd([1,3],[1,4],[1,5],[1,6])=e_0^2$ and $\gcd(e_0,[2,3])=1$.
Therefore $\gcd (E\wedge E')=1$. So we have $\deg\Vc(\Eq)=\deg
E\wedge E'=4n-2$. Notice that $\deg\Res_t(h_1,h_1')=4n-1$.
Therefore we see $\deg G=\deg \Res_t(h_1,h_1')/LC(h_1)=4n-2$, i.e.
$G=0$ is the implicit equation of $\Vc(\Eq)$.
\end{proof}

Thereinafter, we will show that for the curve of general type, we
have $d_w(G)=6n-4$, where $n=\deg\Ep$. For this we  consider
another resultant $\Res_t(h_1,h_2)$ and show that
$d_w(\Res_t(h_1,h_2))=d_w(G)$. We define $h_2$ in the following
way. Let $ g=e_0^2$ then we have the following  equality
$h_1'g-h_1g'=e_0h_2$, where
\begin{equation}
    h_2=(2\langle e,e\rangle e_0'-\langle e,e\rangle' e_0)y_0/2+e_0\langle
    e'e_0-ee_0',y\rangle=(2\langle e,e\rangle e_0'-\langle e,e\rangle' e_0)y_0/2+e_0\langle
    w,y\rangle.   \nonumber
\end{equation}
In other words, $h_2$ is the numerator of a rational function
$\left(\frac{h_1}{g}\right)'$.

We often use the following properties of the resultant
\begin{eqnarray}
\label{res-product}
\Res(f_1f_2,h)&=&\Res(f_1,h)\Res(f_2,h)\ \mbox{(see \cite{cls}, p. 73)},\\
   \label{res-divisor} \Res(f,h)&=&a_0^{m-\deg r}\Res(f,r)\ \mbox{(see \cite{cls}, p. 70),}
\end{eqnarray}
{where} $h=qf+r,\ \deg
    r \leq \deg h=m, f=a_0x^l+a_1x^{l-1}+...+a_l$.

We need an explicit formula for factors of the resultant
\begin{equation}\label{h1g}
   LC(h_1) \Res{}_t(h_1,h_1'g-h_1g')=
\Res{}_t(h_1,h_1'g),
\end{equation}
where $LC(h_1)$ is a leading coefficient with respect to variable
$t$ of the polynomial $h_1$. Indeed, since  $\deg (h_1,t)=\deg
(g,t)$ then $\deg (h_1'g-h_1g',t) +1=\deg (h_1'g,t) $. Thus  we
obtain the  formula in (\ref{h1g}) from the property
(\ref{res-divisor}). The left side of the formula (\ref{h1g}) is
equal to
$LC(h_1)\Res_t(h_1,e_0h_2)=LC(h_1)\Res(h_1,e_0)\Res_t(h_1,h_2)$.
Otherwise, the right side of this formula is equal to
$\Res_t(h_1,h_1') \Res_t(h_1,e_0)^2=G\cdot
LC(h_1)\Res_t(h_1,e_0)^2$.

 Also, from the property (\ref{res-divisor}) follows that
$\Res_t(h_1,e_0)=y_0^{\deg (e_0,t)}$. Therefore we have
 and  $d_w(\Res_t(h_1,h_2))=d_w(G)$. In the lemma below, we prove
 that $d_w(\Res_t(h_1,h_2))=6n-4$.

We summarize our computations in the following
\begin{thm}\label{degree} The degree of the hypersurface $\Gamma(\Ep)$ with the spine  curve $\Ep$
which satisfies the assumption (\ref{assump})  is equal to $6n-4$,
where $n=\deg\Ep$.
\end{thm}

\begin{lem} With the notation as above we suppose that the conditions (\ref{assump}) are
satisfied. Then we have the equality
$d_w(\Res{}_t(h_1,h_2))=6n-4$, where $n=\deg\Ep$. \label{lemma1}
\end{lem}

\begin{proof}
 The weighted degree  $d_w(G)$ may be viewed as a degree of
 a variety $\Phi(\{G\}\cap \Qc)$, where $\Phi :\bp^5\setminus q \to\bp^4$ is a
linear projection from the improper point $q=(1,0,...,0)$ on the
Lie quadric (see explicit formula (\ref{prj})). The degree of the
variety $\Phi (\{G\}\cap \Qc)$ can be computed constructively by
counting  points of intersection with a general line $L\in\bp^4$.
The pre-image of the line $C:=\Phi^{-1}(L)$ is a conic on the Lie
quadric $\Qc$ which passes through the improper point $q$. Hence the
degree of $\Phi (\{G\}\cap \Qc)$
 is $2\deg G - i(q,\{G\}\cap C)$,
 where $i(q,\{G\}\cap C)$ is the multiplicity of the intersection
 $\{G\}\cap C$ at  the point
$q$.

We need a parametric representation of the general conic $q\in
C\subset \Qc $. Assume that the conic $C$ is in the parameterized
plane $P:q+k_1u+k_2v$, where $k_1,k_2\in \R^6$. The plane $P$
intersects a singular cone $\{\langle
x,x\rangle=x_1^2+x_2^2+x_3^2-x_4^2\}$ on  two lines. We choose two
vectors $k_1,k_2$ in these lines so that the first coordinate is
zero, i.e. $k_1=(0,1,a),k_2=(0,1,b),a=(a_1,a_2,a_3,a_4),
b=(b_1,b_2,b_3,b_4) $ such that $\langle a,a\rangle=\langle b, b
\rangle =0$. With the notations as above the general conic
$C:=P\cap Q$ has the following parametrization:
\begin{equation}\label{conic}
    C(u):=(ku-1,ku^2,ku^2a+u(b-a)),\ \mbox{where}\ k = 2 \langle a,b
    \rangle .
\end{equation}
Since  $C(0)=q$ we can compute the multiplicity
$m=i(q,\{\Res(h_1,h_2,t)\}\cap C)$  as follows. Lets denote by
$ch_1=h_1|_C,ch_2=h_2|_C$ the restriction of  polynomials
$h_1,h_2$ to the conic $C$:
\begin{eqnarray}\label{h1conic}
    ch_1&=&-e_0^2(ku-1)/2-k f u^2/2+e_0\langle e,ku^2a+u(b-a) \rangle,  \\
    ch_2&=&u ((2fe_0'-f'e_0)ku/2+e_0\langle
    w,k u a+b-a \rangle), \ \text{where}\ f=\langle e,e\rangle.\label{h2conic}
\end{eqnarray}
The computation of the resultant gives
$\Res{}_t(ch_1,ch_2)=u^m(A+Bu+...)$, here
$m=i(q,\{\Res{}_t(h_1,h_2)\}\cap C)$. On the other side we can
compute the number of common points $(0,t_0)$ on
curves $\{ch_1(u,t)\}$ and $\{ch_2(u,t)\}$ counted with multiplicities. This number coincides
with $m$ (see \cite{buse}, Proposition 5). The second curve
$ch_2(u,t)=u\cdot ch_3(u,t)$ is reducible. Therefore, the
resultant with respect to $t$ is
\begin{eqnarray*}\nonumber
    \Res(ch_1,ch_2)&=& \Res(ch_1,u)\Res(ch_1,ch_3)=u^{\deg (h_1,t)}\Res(ch_1,ch_3)\\
    &=&u^{2n}\Res(ch_1,ch_3)
\end{eqnarray*}

The second factor has a representation $ch_3=uD(t)-N(t)$, where
$D(t)=(2fe_0'-f'e_0)k/2+e_0\langle
    w,k  a \rangle$ and $ N(t)=e_0\langle
    e'e_0-ee_0',a-b \rangle$.
It is easy to see that $\gcd(N(t),e_0^2)=e_0$. If
$ch_1(0,t_0)=ch_3(0,t_0)=0$ then $e_0(t_0)=0$. The first curve
$\{ch_1(u,t)\}$ is hyper-elliptic, i.e.  the projection to $t$
axes $pr:\{ch_1\}\to t$ is a map two-to-one. The curve
$\{ch_1(u,t)\}$ has the following discriminant with respect to
$u$:
\begin{equation}\label{discrim}
   \text{disc}(ch_1,u)=e_0^2(\ (k e_0/2-\langle e,b-a\rangle )^2
    - 2e_0\langle e,k a\rangle +k f\ ).
\end{equation}
It is ease to see that  point $(0,t_0)$ is a singular point on the
curve $\{ch_1(u,t)\}$ if and only if $e_0(t_0)=0$. Therefore the
point $(0,t_0)$ has multiplicity at least two as a point of the
intersection of two curves $\{ch_1\}\cap\{ch_3\}$.

We will prove that the multiplicity of the intersection
 of two curves $\{ch_1\}\cap\{ch_3\}$ at the
point $(0,t_0)$ equals to two if $e_0(t_0)=0$. For simplicity we
assume that $t_0=0$. The first equation (\ref{h1conic}) in the
local ring $R=\R[u,t]_{\langle u,t\rangle}$ is
\begin{eqnarray}
   && \overline{ch}_1=a_{21}u^{2} t+a_{20}u^{2}+a_{12}u{t}^{2}-a_{11}ut+a_{02}{t}^{2},\
     \text{where}  \nonumber \\
&&[ a_{21}, a_{20}, a_{12}, a_{11}, a_{02} ] = [\tilde{e}_{0}k
\langle e, a\rangle ,-f k/2,-{\tilde{e}_{{0}}}^{2}\,k/2,\tilde{
e}_{{0}}\langle e,{a-b}\rangle ,1/2\,\tilde{e}_{{0}}^{2}]\
 \end{eqnarray}
and $t\tilde{e}_0=e_0$. The second equation (\ref{h2conic}) in the
local ring $R$ is
\begin{eqnarray}
 &&  \overline{ch}_2=u(\overline{ch}_3),\quad
 \overline{ch}_3=
 b_{{12}}u{t}^{2}+b_{{11}}ut+b_{{10}}u+b_{{02}}{t}^{2}+b_{{01}}t, \nonumber \\
   && \mbox{where}\quad [b_{{12}},b_{{11}},b_{{10}},b_{{02}},b_{{01}}]=  \nonumber \\
    && [\tilde{e}_{0}^{2}k\langle e', a\rangle,-{f}'
\tilde{e}_{{0}}-\tilde{e}_{0}{e}'_{{0}}k\langle e, a\rangle, k f
\left ({e}'_{{0}}+\,{e}'_{{0}}\right ),- \tilde{e}_{0}^2k \langle
e',{a-b}\rangle ,\tilde{e}_{0}{e}'_{{0}} \langle
e,{a-b}\rangle].\nonumber
\end{eqnarray}

 An easy computation with MAPLE shows that
\begin{equation}
    \Res{}_t(\overline{ch}_1,\overline{ch}_3)=u^2(K_4u^4+K_3u^3+K_2u^2+K_1u+K_0).  \nonumber
\end{equation}
Therefore the point $(0,0)$ has multiplicity two if and only if
$K_0\not=0$, i.e. $K_0$ is a unit in the local ring $R$. A
straightforward computation shows that
\begin{equation}
    K_0=a_{{02}}\left ({b_{{10}}}^{2}a_{{02}}+{b_{{01}}}^{2}a_{{20}}+b_{{01}}b_{{10}}a_{{
11}}\right ) = f k{\tilde{e}_{{0}}}^{4}{{e}'_{{0}}}^{2} \left (k
\langle e,e \rangle +\frac{3}{4}\,{\langle e,a-b\rangle}^{2}\right
).  \nonumber
\end{equation}
Since $f(0)\not=0$ and $e_0'(0)\not=0$  by the condition
(\ref{assump}) we conclude that $K_0\not=0$ for a general conic.
 Hence,  the multiplicity $i(q,\{\Res{}_t(h_1,h_2)\}\cap \Qc)$ is equal
 to
$\deg(h_1,t)+2\deg (e_0,t)=4n$. Therefore, we have
\begin{eqnarray*}\nonumber
d_w(\Res{}_t(h_1,h_2))&=&2\deg(\Res{}_t(h_1,h_2))-i(q,\{\Res{}_t(h_1,h_2)\}\cap
\Qc)\\&=&2(5n-2)-4n=6n-4.
\end{eqnarray*}
\end{proof}


\begin{rem}
  We conjecture  that the degree of the
 hypersurface $\Gamma(\Ep)$ is $\deg \Vc(\Eq)+\deg(w) -\deg(\gcd(w))$,
 where $w$ is defined by the formula (\ref{seq-w}).
\end{rem}


\section{Examples and special cases.}

Let $n$ be the degree of the spine curve $\Ec=\{e(t)\in \R^4\}$.
And let $c_n(\ee)$ be the  degree of the hypersurface
$\Gamma(\Ep)$ with the spine curve $\Ec$.

\vskip0.2cm { Polynomial case.} Assume that the spine curve is
polynomial, i.e. $e_0=1$.
 By the theorem in \cite{xu} the degree of  hypersurface $\Gm(\Ep)$ with
the polynomial spine is at most $4n-2$.

 For the
general spine curve we have $c_n(\ee)=6n-4$, i.e. $c_n(\ee)\leq
6n-4$. The lower bound is not clear. There are examples of spine
curves with the following degrees:
\begin{eqnarray}\label{c2e}
    c_2(\ee)&=&3,4,5,6,8;\nonumber\\
    c_3(\ee)&=&6,7,8,9,10,11,12,14;\nonumber\\
    c_4(\ee)&=&8,9,10,11,12,13,14,15,16,17,18,20.\nonumber
\end{eqnarray}
It seems that there does not exist a spine curve  such that
$c_n(\ee)=6n-5$.

We consider three examples.
\begin{exmp} \label{doubleellipsoid}
Let us consider the following spine curve:
$e(t)=\left(0,0,\frac{8t}{1+t^2},
\frac{3-3t^2}{1+t^2}\right),\\n=2.$ This is a proper parametrization
of an ellipse in $\R^4$. 
 We find the \pl coordinate vector $P=\Eq_{}(t)C\wedge
\Eq_{}'(t)C$ and $q=\gcd(P)=1$. Also, we see that $ \deg \Vc({\Eq}
)=\deg P-\deg(q,t)=8$ and $\gamma=\deg (w)=2$. If we run the
$\mu$-basis algorithm with two input vectors $\Ec_{}(t)C,
\Ec_{}'(t)C$ we get the output two vectors $E_1$ and $E_2$:
\begin{eqnarray}E_1 \cdot \yyq^T  &=& 4\,{t}^{3}y_{{3}}+\left (-u-41\,y_{{0}}\right ){t}^{2}+12\,ty_{{3}}+9
\,y_{{0}}-u-6\,y_{{4}}, \nonumber\\
E_2 \cdot \yyq^T &=&\left (u-9\,y_{{0}}-6\,y_{{4}}\right
){t}^{3}-12\,{t}^{2}y_{{3}}+ \left (41\,y_{{0}}+u\right
)t-4\,y_{{3}}.  \nonumber
\end{eqnarray}
Now we can find the implicit equation $G=\Res(
\E_1\cdot\hat{y}^T,\E_2\cdot\hat{y}^T,t)$ of the dual variety
$\Vc({\Eq} )$. The polynomial $G$ contains 26 monomials and has
degree 6 (as in Proposition~\ref{degree_of_G}). The equation of
the hypersurface $\Gamma(\Ep)$ is defined by the polynomial\\
$F(y_0,...,y_4) =y_0^2G(\langle
y,y\rangle/y_0,y_0,y_1,y_2,y_3,y_4)$ of degree 8. Since,
\begin{eqnarray}F(1,y_1,y_2,y_3,0)= \left
({y_{{1}}}^{2}+16+{y_{{2}}}^{2}+8\,y_{{3}}+{y_{{3}}}^{2}\right )
\left
({y_{{1}}}^{2}+{y_{{2}}}^{2}+16-8\,y_{{3}}+{y_{{3}}}^{2}\right )\nonumber\\
\left (-225+25\,{y_{{1}}}^{2}+25\,{y_{{2}}}^{2}+9\,{y_{{3}}}^{2}
\right )^{2},\nonumber\end{eqnarray}
 the 0-envelope of the canal surface
${\Env}_0(\ee)=\Gm(\Ep)\cap\{y_4=0\}$ is reducible. The  canal
surface $\Cc$ is the double ellipsoid of revolution
$(-225+25\,{y_{{1}}}^{2}+25\,{y_{{2}}}^{2}+9\,{y_{{3}}}^{2})^2$.
Indeed, for the computation of ${\Cc}(\Ep)$ we should assume that
the variable $y_4=0$ and to repeat the  same steps as above. We
should consider only the first 5 coordinates of the vectors
$\Ec_{}(t)C, \Ec_{}'(t)C$. Let us denote these two vectors with 5
coordinates by ${D}_1,{D}_2$. But this time we see that the \pl
vector $\hat{P}={D_1}\wedge{D_2}$ has a
 non-trivial common divisor, i.e. $\hat{q}=\gcd(\hat{P})=t^2-1$.
 So, using the $\mu$-basis algorithm we find the $\mu$-basis $R_1,R_2$ for the input
 ${D_1},{D_2}$. In this case we see that $\deg R_1=\deg
 R_2=2$.
Now we find the resultant $\check{G}=\Res{}_t(R_1 \cdot \check{y}^T,R_2 \cdot \check{y}^T)=\left
(16\,{y_{{3}}}^{2}+225\,{y_{{0}}}^{2}-25\,y_{{0}}u\right )^{2}$,
where $\check{y}=(u,y_0,y_1,y_2,y_3)$. After the substitution
$u=(y_1^2+y_2^2+y_3^2)/y_0$ we obtain the implicit equation of the
canal surface the double ellipsoid
$(-225+25\,{y_{{1}}}^{2}+25\,{y_{{2}}}^{2}+9\,{y_{{3}}}^{2})^2$.
We can see this geometrically, too. The point $e(t)\in\R^4$
corresponds to the sphere $S(e(t))\in\R^3$ with a center on the
$y_3$-axis. If $t\in [-1/2,1/2]$ then the  sphere  $S(e(t))$ is
tangent to the ellipsoid
$EL=(-225+25\,{y_{{1}}}^{2}+25\,{y_{{2}}}^{2}+9\,{y_{{3}}}^{2})$,
and inside this ellipsoid. Moreover, the real envelope of the
family $S(e(t)),t\in  [-1/2,1/2]$ is the ellipsoid $EL$. Note that
the  sphere $S(e(1/t))$  has the same center but the opposite
radius to the sphere $S(e(t))$, i.e. it has the opposite
orientation. Therefore, the real envelope of the family
$S(e(t)),t\in (-\infty,-2]\cap [2,\infty)$ is the same  ellipsoid
$EL$. Hence, from the point of Laguerre geometry the envelope of
the whole family $S(e(t))$ is the double ellipsoid $EL^2$. Note,
that the d-offset 
to the canal surface, in this case is the d-offset to ellipsoid
$EL$ and it has degree 8. Also, we can check that by
Theorem~\ref{degree} the degree of the $\Gm(\Ep)$ hypersurface is
8, too. For a detailed study and other examples of canal surfaces
with a quadratic spine curve we recommend to look at the paper
\cite{KZ}.
\end{exmp}


\begin{exmp}
 Consider the polynomial spine curve
$e(t)= \left({ {3t^2+1}, {4{t}^{2}+{t}},0, 5{t}^{2}} \right),$
$n=2.$  We find the \pl coordinate vector $P=\Eq(t)C\wedge
\Eq'(t)C$ and $q=\gcd(P)=1$. Also, we see that $ \deg \Vc({\Eq}
)=\deg P-\deg q=4$ and $\gamma=\deg (w)=1$. If we run the
$\mu$-basis algorithm with two input vectors $\Eq_{}(t)C,
\Eq'_{}(t)C$ we get the output of two vectors
\begin{eqnarray}
E_1&=&-u+\left (-1-7\,{t}^{2}-8\,{t}^{3}\right )y_{{0}}+\left
(2+6\,{t}^{2} \right )y_{{1}}+2\,t\left (1+4\,t\right
)y_{{2}}-10\,{t}^{2}y_{{4}} ,\nonumber\\
  E_2&=&-t\left (7+12\,t\right )y_{{0}}+6\,ty_{{1}}+\left (1+8\,t\right )y_{{2
}}-10\,ty_{{4}}
 ,\nonumber
   \end{eqnarray}
 and find the
implicit equation $G$ of the dual variety $\Vc({\Eq} )$ (it
contains 54 monomials, so we do not present an explicit formula).
Finally, we find that $c_2(\ee)=d_w(G)=5$.
 For this example, we have
$\deg{\Cc}=\deg\Gm{(\Ep)}$, i.e. the implicit degree of the canal
surface is 5. Note that this contradicts Theorem 4 in
\cite{xu}, because in this example the degree of the canal surface
is an odd number. It seems that the mentioned theorem gives only an
upper bound estimation, but not the exact degree formula of  canal
surfaces with polynomial spine curve.
\end{exmp}

\begin{exmp}
In the next example we take the following spine curve:\\ $e(t)=
\left( {\frac {\left (1-{t}^{2}\right )^{2}}{\left
(1+{t}^{2}\right )^{2}}}, 2\,{\frac {t\left (1-{t}^{2}\right
)}{\left (1+{t}^{2}\right )^{2}}},2 \,{\frac {t}{1+{t}^{2}}},1
\right), n=4$. The first three coordinates define the Viviani
curve, i.e. it is intersection curve of the sphere and the tangent
cylinder.
We find the \pl coordinate vector $P=\Eq(t)C\wedge \Eq'(t)C$ and
$q=\gcd(P)=1$. Also, we see that $ \deg \Vc({\Eq} )=\deg P-\deg
q=6$. If we run the $\mu$-basis algorithm with two input vectors
$\Eq_{}(t)C, \Eq'_{}(t)C$ we get output of two vectors
\begin{eqnarray}
E_1&=&\left(
0,4+4\,{t}^{2},4-4\,{t}^{2},6\,t-2\,{t}^{3},6\,t+2\,{t}^{3},4+4\,{t}
^{2}\right),
 \nonumber\\
E_2&=&\left( 0,4\,t+4\,{t}^{3},4\,t\left (-1+{t}^{2}\right
),2-6\,{t}^{2},2+ 6\,{t}^{2},4\,t+4\,{t}^{3}\right),
   \nonumber
   \end{eqnarray}
 both of degree 3  and find the
implicit equation $G$ of the dual variety $\Vc({\Eq} )$ (it
contains 58 monomials). Finally, we find that
$c_4(\ee)=d_w(G)=10$.
For this example, we have
$\deg \Cc=\deg {\Gm}(\Ep)$, i.e. the implicit degree of
the canal surface is 10.
\end{exmp}


\end{document}